\documentclass[12pt,oneside]{amsart} 
\usepackage{amssymb,amscd}
\usepackage{euscript}

      \theoremstyle{plain}
      \newtheorem{theorem}{Theorem}[section]
      \newtheorem{lemma}[theorem]{Lemma}
      \newtheorem{corollary}[theorem]{Corollary}
      
      \newtheorem{remark}[theorem]{Remark}
      
      \newtheorem{definition}[theorem]{Definition}      
      \newtheorem{assumptions}[theorem]{Assumptions}     
\numberwithin{equation}{section}

      \makeatletter
      \def\@setcopyright{}
      \def\serieslogo@{}
      \makeatother

\def\M{\EuScript{M}}
\def\c{\EuScript{C}}
\def\R{\mathbb R}
\def\Rm{\mathbb R ^m}

\def\Z{\mathbb Z}

\def\N{\mathbb N}

\def\G{\mathcal G}

\def\S{\mathcal S}
\def\po{{\mathcal{R}}}
\def\s{{\mathcal{S}}}
\def\c{{\mathcal{C}}}
\def\n{{\mathcal{N}}}

\def\rm{\Lambda ^\mu}

\newcommand{\la}{\lambda}
\newcommand{\La}{\Lambda}

\def\dist{\text{dist}}

\def\Id{\text{Id}}
\def\e{\varepsilon}
\def\a{\alpha}
\def\b{\beta}

\def\Ci{C^\infty}

\def\Cr{C^{N,\alpha}}

\def\bw{\bar{W}}

\def\f{\bar f}

\def\QED{\hfill\hfill{\square}}

\def\E{{\mathcal{E}}}
\def\Ex{{\mathcal{E}_x}}
\def\V{{\mathcal{V}}}

\def\bb{{\mathcal{B}}}
\def\bx{{\mathcal{B}^{x}}}
\def\qx{{\tilde Q}}

\def\fe{{\mathcal {F}}}
\def\f{{\mathcal {F}}}
\def\fd{{F}}

\def\pe{{\mathcal {P}}}
\def\p{{\mathcal {P}}}
\def\pd{{P}}

\def\h{{\mathcal {H}}}
\def\hd{H}

\def\g{\mathcal G}
\def\G{\Gamma}

\begin{document}

\date{\today}
\author{Boris Kalinin and Victoria Sadovskaya$^{\ast}$}

\address{Department of Mathematics, The Pennsylvania State University, University Park, PA 16802, USA.}
\email{kalinin@psu.edu, sadovskaya@psu.edu}

\title [Normal forms for non-uniform contractions]
{Normal forms for non-uniform contractions} 

\thanks{$^{\ast}$ Supported in part by NSF grant DMS-1301693}



\begin{abstract}
Let $f$ be a measure-preserving transformation of a Lebesgue space $(X,\mu)$
and let  $\f$ be its extension  to a bundle $\E = X \times\Rm$ by smooth 
fiber maps $\f_x : \E_x \to \E_{fx}$ so that the derivative of $\f$ at the zero 
section has negative Lyapunov exponents.
We construct a measurable system of smooth coordinate changes $\h_x$ on $\E_x$ for $\mu$-a.e. 
$x$ so that the maps $\p_x =\h_{fx} \circ \f_x \circ \h_x ^{-1}$
are sub-resonance polynomials  in a finite dimensional Lie group. 
Our construction shows that such  $\h_x$ and $\p_x$ are unique
up to a sub-resonance polynomial. As a consequence, we obtain the centralizer 
theorem that the coordinate change $\h$ also conjugates any commuting 
extension to a polynomial extension of the same type.
We apply our results to  a measure-preserving diffeomorphism $f$ with a 
 non-uniformly contracting invariant foliation $W$. We construct a measurable 
 system of smooth coordinate changes $\h_x: W_x \to T_xW$ such 
 that the maps $\h_{fx} \circ f \circ \h_x ^{-1}$ are polynomials
 of sub-resonance type. Moreover, we show that for almost every 
leaf the coordinate changes exist at each point on the leaf and 
give a coherent atlas with transition maps in a finite dimensional Lie group. 

\end{abstract}

\maketitle


\section{Introduction}

The theory of normal forms for smooth maps originated in the works of Poincare 
and Sternberg \cite{St} and normal forms at fixed points and invariant manifolds have 
been extensively studied \cite{BK}. More recently, non-stationary normal form theory 
was developed in the context of a diffeomorphism $f$ contracting a foliation $W$. 
The goal is to obtain a family 
of diffeomorphisms $\h_x: W_x \to T_xW$ such that the maps 
\begin{equation} \label{form}
 \tilde f_x =\h_{fx} \circ f \circ \h_x ^{-1}: \;T_x W \to T_{fx}W
\end{equation}
are as simple  as possible, for example linear maps or polynomial maps in a finite dimensional Lie group. Such a map $\tilde f_x$ is called a normal form of $f$ on $W_x$.

The non-stationary normal form theory started with the linearization along 
one-dimensional foliations obtained by Katok and Lewis \cite{KL}. In a more general
setting of contractions with narrow band spectrum, it was developed by Guysinsky 
and Katok \cite{GK,G}, and  a differential geometric point of view was presented 
by Feres \cite{F2}. For the  linearization, further results  were obtained
by the second author in \cite{S} and it was shown in \cite{KS} that the coordinates 
$\h_x$ give a consistent affine atlas on each leaf of $W$. 
In \cite{KS15} we extended these results to the general narrow band case.
More precisely, 
we gave a construction of $\h_x$ that depend smoothly on $x$ along the leaves
and proved that they define an atlas with transition maps in a finite dimensional
Lie group.  Non-stationary normal forms were used extensively in 
the study of rigidity of uniformly hyperbolic dynamical systems and group actions, 
see for example \cite{KSp97,KS03,KS,Fa,FFH,GKS10, FKS}.  

To obtain applications for non-uniformly hyperbolic systems and actions, one 
needs a similar theory of non-stationary normal forms for non-uniform contractions. 
The existence and centralizer theorems were stated without proof in \cite{KKt00}
along with a program of potential applications. The theory, however, 
was not developed for quite a while. The linearization of a $C^{1+\a}$ diffeomorphism along a one-dimensional non-uniformly contracting 
foliation  was constructed  in \cite{KKt} and used in 
the study of measure rigidity in \cite{KKt,KKtR}. Similar results for higher 
dimensional foliations with  pinched exponents were obtained 
by Katok and Rodriguez Hertz in \cite{KtR}. The existence of $\h_x$ for 
a general $C^\infty$ extension was proved by Li and Lu \cite{LL}.

In this paper we develop the theory of non-stationary polynomial normal 
forms for smooth extensions of measure preserving transformations by 
non-uniform contractions, described in the beginning of Section 2. 
This is a convenient general setting for the
construction. The foliation setting reduces to it by locally identifying 
the leaf $W_x$ with its tangent space $\E_x=T_xW$ and viewing
$\f _x = f|_{W_x} : \E_x \to \E_{fx}$ as an extension of the base system 
$f: \M \to \M$ by smooth maps on the bundle $\E=TW$. The base system
can then be viewed as just a measure preserving one. In the 
extension setting, the map $\h_x$ is a coordinate change on $\E_x$
and we denote 
$$\p_x=\h_{fx} \circ \f_x \circ \h_x ^{-1}: \E_x  \to \E_{fx}.
$$ 
In Theorem  \ref{NFext} we construct coordinate changes $\h_x$ for 
$\mu$ almost every $x$  so that $\p_x$ is a sub-resonance polynomial. 
For any regularity of $\f$ above the critical level, we obtain $\h$
in the same regularity class.

Our construction allows us to describe the exact extent of non-uniqueness 
in $\h_x$ and $\p_x$. Essentially, they are defined up to a sub-resonance polynomial. As a consequence of this, we obtain the centralizer 
theorem that the coordinate change $\h$ also conjugates any commuting 
extension to a normal form of the same type. We just learned of similar 
results  in differential geometric formulations  by Melnick \cite{M}.
The approach in \cite{M}  is different from ours and it relies on 
ergodic theorems for higher jets of $\f_x$. 
Our results  assume only temperedness of the  higher derivatives 
of $\f_x$ rather than certain integrability required in  \cite{M}. This allows 
 us to obtain applications to the foliation setting without any 
 assumptions on transverse regularity of the foliation.
 
In particular, we consider a diffeomorphism $f$ which preserves an 
ergodic measure with some negative Lyapunov exponents and take $W$
to be any strong part of the stable foliation.  In this setting Theorem \ref{NFfol} 
gives sub-resonance normal forms for $f$ along the leaves of $W$.  
Moreover, we show that for almost every leaf the normal form 
coordinates $\h_x$ exist at each point on the leaf and give a coherent 
atlas with transition maps in a finite dimensional Lie group $G$ determined 
by sub-resonance polynomials. 
This yields an invariant  structure of a $G$ homogeneous space on almost every 
leaf. 

We expect these results to be useful in the study of non-uniformly 
hyperbolic systems and group actions.


\section{Statements of results}\label{Snormalforms}

\begin{assumptions} \label{ass} In this paper, \\
$(X,\mu)$ is a  Lebesgue probability space, \\
$f:X\to X$ 
is an invertible ergodic measure-preserving transformation of $(X,\mu)$,\\
$\E = X \times\Rm$ is a finite dimensional vector bundle over $X$,\\
$\V$ is a neighborhood of the zero section in $\E$, \\
$\fe:\V  \to \E$ is a measurable extension of $f$ that preserves the zero 
 section,  \\
 $F:\E\to \E$ is the derivative of  $\fe$ at zero section, 
 $F_x=D_0\fe_x :\E_x \to \E_{fx}$,\\
 $F$ and $F^{-1}$ exist and satisfy  $\,\log\|F_x\|\in L^1(X,\mu)$ and $\,\log\|F_x^{-1}\|\in L^1(X,\mu)$,\\
and the Lyapunov exponents of $F$ are negative: $\chi_1<\dots<\chi_\ell<0$.
 \end{assumptions}

\noindent  {\bf Sub-resonance polynomials.} 
 Let $\chi_1<\dots<\chi_\ell<0$ be the distinct Lyapunov exponents of $\fd$ 
and let $\Ex=\E_x^{1} \oplus \dots \oplus \E_x^{\ell}$ be the splitting of $\E_x$ 
for $x \in \La$ into the Lyapunov subspaces given by the Multiplicative Ergodic Theorem \ref{MET}.

We say that a map between vector spaces is {\em polynomial}\, if each component  is given by a polynomial 
in some, and hence every, bases.
We consider a polynomial map $P: \E_x \to \E_y$ with $P(0_x)=0_y$ and 
split it into components $(P_1(t),\dots,P_{\ell}(t))$, where $P_i: \E_x \to \E_y^i$. 
Each $P_i$ can be written uniquely as a linear combination of  polynomials
of specific homogeneous types: we say that $Q: \E_x \to \E_y^i $ has homogeneous type $s= (s_1, \dots , s_\ell)$ if for any real numbers 
$a_1, \dots , a_\ell$ and vectors
$t_j\in \E_x^j$, $j=1,\dots, \ell,$ we have 
\begin{equation}\label{stype}
Q(a_1 t_1+ \dots + a_\ell t_\ell)= a_1^{s_1} \cdots  a_\ell^{s_\ell} \, Q( t_1+ \dots + t_\ell).
\end{equation}

\begin{definition} \label{SRdef}
We say that a polynomial map $P: \E_x \to \E_y$ is {\em sub-resonance} 
if each component $P_i$ has only terms of homogeneous types $s= (s_1, \dots , s_\ell)$  satisfying {\em sub-resonance relations}
\begin{equation}\label{sub-resonance}
\chi_i\le\sum s_j \chi_j, \quad\text{where
$s_1,\dots,s_{\ell}\,$ are non-negative integers.}
\end{equation}
We denote by $\s_{x,y}$ the space of all sub-resonance polynomial maps 
from $\E_x$  to $\E_y$.
\end{definition}
Clearly, for any sub-resonance relation we have $s_j=0$ for $j<i$ and $\sum s_j \le \chi_1/ \chi_\ell$.
It follows that sub-resonance polynomial maps have degree at most 
\begin{equation}\label{degree}
d=d(\chi)= \lfloor \chi_1/\chi_\ell \rfloor.
\end{equation}
Sub-resonance polynomial maps $P: \E_x \to \E_x$ with $P(0)=0$ with invertible derivative at the origin form a group with respect to composition \cite{GK}. 
We will denote this  finite-dimensional Lie group by $G_x^{\chi}$. 
All groups $G_x^{\chi}$ are isomorphic, moreover,
any map $P\in \s_{x,y}$ with $P(0_x)=0_y$ and invertible derivative 
at $0_x$  induces an isomorphism between $G_x^{\chi}$ and $G_y^{\chi}$
by conjugation.


We denote by $B_{x,\sigma (x)}$ the closed ball of radius $\sigma (x)$ centered at 
$0 \in \E_x$. For $N\ge1$ and $0<\a \le1$ we denote by 
$\Cr (B_{x,\sigma (x)})=\Cr (B_{x,\sigma (x)}, \E_x)$ 
the space of functions from $B_{x,\sigma (x)}$ to 
$\E_x$ with continuous derivatives up to order $N\ge1$ on $B_{x,\sigma (x)}$ 
and with $N^{th}$ derivative satisfying $\a$-H\"older condition at $0$:
\begin{equation}\label{Canorm}
\| D^{(N)} R\|_\a =\sup \,\{ \,\| D^{(N)}_t R - D^{(N)}_0 R\|\cdot \|t\|^{-\a} : \; 0\ne t \in B_{x,\sigma (x)}\} < \infty.
  \end{equation}
 We call $\| D^{(N)} R\|_\a$ the $\a$-H\"older constant of $D^{(N)} R$ at $0$.
We equip the space $\Cr (B_{x,\sigma (x)})$ with the norm
\begin{equation}\label{Crnorm}
  \|R\|_{\Cr (B\, x,\sigma (x))}=\,
  \max\, \{\, \|  R\|_0, \;\| D^{(1)} R\|_0,\; ..., \;\| D^{(N)} R\|_0, \;\| D^{(N)} R\|_\a \, \},   \end{equation}
where $\| D^{(k)} R\|_0=\sup \{ \| D^{(k)}_t R\|: \; t \in B_{x,\sigma (x)}\}$.

  \vskip.1cm
  
 We say that a non-negative real-valued function $K$ on $X$ is 
{\it $\e$-tempered at $x$} if 
\begin{equation} \label{e-temp}
 \sup \,\{K(f^n x)\, e^{-\e n} :\, {n \in \N}\} < \infty,
\end{equation}
and that  $K$ is  {\it $\e$-tempered on a set}\, if it is $\e$-tempered at each of its points.

\vskip.1cm


We consider an extension $\f$ satisfying the Assumptions \ref{ass} and
denote by $\La$ the set of regular points and by $\chi_1<\dots<\chi_\ell<0$ the Lyapunov exponents of $F$ given by the Multiplicative Ergodic Theorem \ref{MET}.
For $N$ and $\a$ as above we define 
\begin{equation}\label{kappa}
\kappa = 1+3/\a \,\text{ if }\, N=1\;\text{ and  }\;
\kappa = 4 \,\text{ if }\, N\ge2 . 
\end{equation}
If $N\ge 2$ we allow $\a=0$, in which case we understand $\Cr$ as $C^N$.

\begin{theorem}[Normal forms for
non-uniformly contracting extensions]\label{NFext} $\;$ \\
Let $\f$ be an extension of $f$ satisfying Assumptions \ref{ass}. Suppose that 
\begin{equation}\label{crit}
N\ge1, \quad 0\le \a \le1 \quad \text {and} \quad N+\a> \chi_1 / \chi_\ell .
\end{equation}
Then there exist positive constants $L=L(N,\a)$ and  
$\e_*=\e_* (N,\a, \chi_1 , ... , \chi_\ell)$ so that for any $\,0<\e\le \e_*$  the following holds. 
 \vskip.1cm
 
 If there exists a positive measurable function $\sigma:\La \to  \R$ so that 
$1/\sigma$ is $\e$-tempered on $\La$ and $\f_x$ is $\Cr(B_{x,\sigma(x)})$ for all 
$x\in \La$ with the derivatives measurable in $x$ and with $\|\f_x \|_{\Cr}$ 
 $\e$-tempered on $\La$ then 
\vskip.2cm

\noindent {\bf (1)}  There exists a positive measurable  function 
$\rho :\La \to  \R$  which is $\kappa \e$-tempered on $\La$ and a measurable family $\{ \h_x\}_{x\in \La}$ of $\Cr$ diffeomorphisms   
 $\h_x : B_{x,\rho (x)} \to \E_x$  satisfying $\h_x(0)=0$ and $D_0 \h_x =\Id \,$
 which conjugate $\fe$ to a sub-resonance polynomial extension $\pe$: 
 $$\h_{fx} \circ \f_x =\p_x \circ \h_x, \; \text{ where }  \;  \p_x\in \S_{x,fx}
 \; \text{ for all } x\in \La.$$
 Moreover, $\|\h_x \|_{\Cr(B\,x,\rho(x))}$ is $L\e$-tempered 
 on $\La$ and $\|D^{(n)}_0\h_x\|$ is $n^2\e$-tempered on $\La$  for $n=1,...,N$, 
with respect to the $\e$-Lyapunov metric \eqref{Lprod}.

\vskip.2cm

\noindent {\bf (2)} Suppose  $\tilde \h=\{ \tilde \h_{x}\}_{x\in \La}$ is another 
measurable family of diffeomorphisms as in (1) conjugating $\f$ 
to a sub-resonance polynomial extension  $ \tilde \p$. 
Then for all $x \in \La$ there exists $G_x \in G_x^\chi$ 
which is measurable and tempered in $x$ such that 
$\h_x=G_x \circ \tilde \h_{x}$. Moreover, if 
$D^{(n)}_0\tilde \h_x=D^{(n)}_0\h_x$ for all $n=2,...,d= \lfloor \chi_1/\chi_\ell \rfloor$, then
$\h_x= \tilde \h_{x}$ for all $x \in \La$. In particular, $\{ \h_x\}_{x\in \La}$
is unique if $\,d=1$.

\vskip.2cm

\noindent  {\bf (3)} Let $g:X\to X$ be an invertible map commuting with $f$ and
let $\La'$ be a subset of $\La$ which is both $f$ and $g$ invariant.
Let $\g (x,t)=(g(x),\g_x(t))$ be an extension of $g$ to $\E$
which preserves the zero section and commutes with $\f$. 
Suppose that $\g_x$ is $\Cr(B_{x,\sigma(x)})$ for all $x\in \La'$ with 
the derivatives measurable in $x$, and that $\|\g_x \|_{\Cr}$ and
$\|(D_0\g_x)^{-1} \|$ are $\e$-tempered on $\La'$. Then
$ \h_{gx} \circ \g_x \circ \h_x^{-1} \in\ S_{x,fx}$ for all $x\in \La'$.

\end{theorem}

\begin{corollary} \label{Cinf}
Suppose that $\f_x$ is $\Ci(B_{x,\sigma(x)})$ and that $1/\sigma$ and 
$\|\f_x \|_{C^N}$ for each $n\in \N$ are $\e$-tempered on $\La$ for each 
$\e>0$. Then $\h_x$ in part (1) of Theorem \ref{NFext}  is $\Ci (B_{x,\rho (x)})$.
\end{corollary}

\vskip.2cm 
\noindent {\bf Normal forms on stable manifolds}.
Let $\M$ be a smooth manifold and let $f$ be a diffeomorphism of $\,\M$ 
preserving an ergodic Borel probability measure $\mu$.
We assume that $f$ is $\Cr$, that is $C^N$ with $N^{th}$ derivative 
$\a$-H\"older on $\M$.
We denote by $\La$ the full measure set of Lyapunov regular points for $(f,\mu)$.
Let $\chi_1<\dots<\chi_{\ell'}$ be the Lyapunov exponents of $(f,\mu)$ and 
suppose $\ell$ is such that $\chi_{\ell}<0$.  Then for each $x\in \La$ there exists
the (strong) stable manifold $W_{x}$  tangent to $\E_x=\E^1_x \oplus ... \oplus \E^\ell_x$\, \cite[Theorem 6.1]{R}.
 
\begin{theorem} [Normal forms on stable manifolds] \label{NFfol} 
Let $\M$ be a smooth manifold and let $f$ be a $\Cr$ diffeomorphism of 
$\,\M$  preserving an ergodic Borel probability measure $\mu$. 
 Suppose that $N\ge 1$, $0<\a\le1$ and $N+\a>  \chi_1/\chi_\ell$.
Then there exist a full measure set $X$ which consists of full stable 
manifolds $W_x$ and
a measurable family $\{ \h_x \} _{x\in X}$ of $\Cr$ diffeomorphisms  
 $$\h_x: W_x \to \E_x =T_xW_x \quad \text { such that }$$

\noindent  {\bf (i)} $\;\,\p_x =\h_{fx} \circ f \circ \h_x ^{-1}:\,\E_{x} \to \E_{fx}\,$ is a sub-resonance polynomial map for each $x \in X$,
\vskip.15cm

\noindent  {\bf (ii)}  $\,\h_x(x)=0$ and $D_x\h_x $ is the identity map  for each $x \in X$,
\vskip.15cm

\noindent  {\bf (iii)} $\| \h_x\|_{\Cr}$ is tempered on $X$,
 \vskip.15cm 
 
\noindent  {\bf (iv)} $\,\h_y \circ \h_x^{-1} : \E_x \to \E_y$ 
is a sub-resonance polynomial map for all $x \in X$ and $y \in W_x$,
 \vskip.15cm 
 
\noindent  {\bf (v)} If $g:\M\to \M$ is a $\Cr$ diffeomorphism commuting with $f$ 
which preserves the measure class of $\mu$ then
$ \h_{gx} \circ \g_x \circ \h_x^{-1} : \E_{x} \to \E_{gx}$ is a sub-resonance polynomial map  for all $x$ in a full measure set $X'$ which consists of full stable 
manifolds.

\end{theorem}
 
Another way to interpret (iv) is to view $\h_x$ as a coordinate chart on 
$W_x$ identifying it with $\E_x$. 
 In this coordinate chart, (iv) yields that all transition maps $\h_y \circ \h_x^{-1}$ for $y\in W_x$ are in  the group $\bar G_x^\chi$ generated by $G_x^\chi$ and 
 the translations of $\E_x$. Thus $\h_x$ 
gives the leaf a structure of homogeneous space $W_x\sim \bar G_x^\chi/G_x^\chi$,  
which is consistent with other coordinate charts $\h_y$ for $y\in W_x$ 
and is preserved by the normal form $\p _x$ by (i).

 \begin{corollary} \label{1/2pinch}
 Under the assumptions of the Theorem \ref{NFfol}, if $d= \lfloor \chi_1/\chi_\ell \rfloor =1$, i.e. $2\chi_\ell<\chi_1$, then  $\p_x$ is the linear map $Df|_{\E_x}$, 
  the family $\{ \h_x \} _{x\in X}$ satisfying (ii) and (iii) is unique,
 the maps $\h_y \circ \h_x^{-1} : \E_x \to \E_y$ are affine for all $x \in X$ and 
$y \in W_x$, and $\h_y$ depends $C^N$-smoothly on $y$ along the stable manifolds.
 \end{corollary}





\section{ Lyapunov exponents and Lyapunov norm}\label{preliminaries}

In this section we review some basic definitions and facts of the 
Oseledets theory of linear extensions. We use \cite{BP} as a general reference.
For  a linear extension $F$ of a map $f$ we will use the notation
\begin{equation}  \label{gen+}
F^n_x= F_{f^{n-1}x} \circ \cdots \circ F_{fx} \circ  F_x.
\end{equation}

\begin{theorem} [Oseledets Multiplicative Ergodic Theorem, see \cite{BP} Theorem 3.4.3] 
\label{MET} $\;$ \\
Let $f$ be an invertible ergodic measure-preserving transformation of a Lebesgue probability space $(X,\mu)$. Let $F$ be a measurable linear extension satisfying $\log\|F_x\|\in L^1(X,\mu)$ and $\log\|F_x^{-1}\|\in L^1(X,\mu)$. Then there exist numbers $\chi_1 < \dots < \chi_{\ell}$, an $f$-invariant set $\La$ with $\mu (\La)=1$, and an $F$-invariant Lyapunov decomposition 
$$
\E_x= \E^{1}_x\oplus\dots\oplus \E^{\ell}_x\; \text{ for }x\in \La
$$
 such that 
 \vskip.1cm
 \begin{itemize}
\item[(i)] $ \underset{n\to{\pm \infty}}{\lim} n^{-1} \log\| F^n_x v \|=  \chi_i\, $  for any $i=1,...,\ell$ and any $\,0 \not= v\in \E^{i}_x$, and
 \vskip.1cm
\item[(ii)]  $ \underset{n\to{\pm \infty}}{\lim} n^{-1}  \log\det F^n_x= \sum_{i=1}^{\ell}m_i  \chi_i $,
where $m_i=\dim \E^{i}_x$. 
\end{itemize}
\end{theorem}

\noindent The numbers $\chi_1,\dots,\chi_{\ell}$
are called the {\it Lyapunov exponents} of $F$ and 
the points of the set $\rm$
are called {\it regular}. 

\vskip.2cm


We denote the standard scalar product in $\Rm$ by $\langle\cdot,\cdot \rangle$. For a fixed $\e >0$ and a regular point $x$,   the {\it $\e$-Lyapunov 
scalar product (or metric)}  $\langle \cdot,\cdot \rangle_{x,\e}$ in $\E_x=\Rm$ 
is defined as follows.
For $u\in \E^i_x$ and $v\in \E^j_x$ with $i\neq j,\,$  
 $\langle u,v \rangle_{x,\e} :=0$, and for $i=1,\dots,\ell$ and $u,v\in \E^i_x,$  
\begin{equation}  \label{Lprod}
\langle u,v \rangle_{x,\e}  =m \,\sum_{n\in\Z}\, \langle  F_x^n (u) ,
F_x^n(v) \rangle \,\exp(-2\chi_i n -\e |n|). 
\end{equation}
Note that the series converges exponentially for any regular $x$. The constant $m$ in front of
the conventional formula is introduced for more convenient comparison with 
the standard scalar product. 
Usually,  $\e$ will be fixed and we will denote $\langle \cdot,\cdot \rangle_{x,\e}$
simply by $\langle \cdot,\cdot \rangle_x$ and call it the {\it Lyapunov scalar product}. The norm 
generated by this scalar product is called the {\em Lyapunov norm}
and is denoted by $\|\cdot\|_{x,\e}$ or $\|\cdot\|_x$.

\vskip.1cm

Below we summarize the basic properties of the Lyapunov scalar product 
and norm, for more details see  \cite[Sections 3.5.1-3.5.3]{BP}. A direct calculation 
shows \cite[Theorem 3.5.5]{BP} that for any regular $x$ and any $u\in \E^i_x$ 
\begin{equation}  \label{estAEi}
\exp(n \chi_i -\e |n|) \,\|u\|_{x,\e} \le 
\| F_x^n(u)\|_{f^n x,\e} \le
\exp(n \chi_i+\e |n|)\,\|u\|_{x,\e} \quad \text{for all }n\in \Z,
\end{equation}
\begin{equation}  \label{estAnorm}
\exp(n \chi_\ell -\e n) \le \| F^n_x \|_{f^n x \leftarrow x} \le \exp(n \chi_\ell+\e n) 
\quad \text{for all }n \in \N,
\end{equation}
where 
$\| \cdot \|_{f^n x \leftarrow x}$ is the operator norm with respect to the
Lyapunov norms. It is defined for any points $x,y\in \La$ and any 
linear map $F:\E_x \to \E_y$ as follows
$$
\| F \| _{y\leftarrow x}=\sup \, \{ \| Fu \|_{y,\e}:  \; u\in \E_x, \;\, \| u\|_{x,\e}=1 \}.
$$
We emphasize that Lyapunov scalar product and norm are defined only for regular points and depend  measurably on the point. Thus, a comparison 
with the standard norm is important. The uniform lower bound follows easily from the definition: $\|u\|_{x,\e}\ge \|u\|$. The upper bound is not uniform,
but it changes slowly along the regular orbits 
\cite[Proposition 3.5.8]{BP}: there exists a 
measurable function $K_\e (x)$ defined on the set of regular points $\La$ 
such that 
\begin{equation}  \label{estLnorm}
\| u \| \le  \| u \|_{x,\e} \le K_\e(x) \|u\| \quad \text{for all } x \in \La 
\text{ and } u \in \E_x, \quad \text{and} 
\end{equation}
\begin{equation}  \label{estK}
 K_\e(x) e^{-\e |n|}  \le K_\e(f^n x) \le  K_\e(x) e^{\e |n|}  \quad 
\text{for all }  x \in\La \text{ and } n \in \Z.
\end{equation}
These estimates are obtained in \cite{BP} using the fact that 
$\|u\|_{x,\e}$ is {\em tempered}, but they can also be verified directly using the definition of $\|u\|_{x,\e}$ on each Lyapunov space and noting that angles between the spaces change slowly. 

Using \eqref{estLnorm} we obtain that for any point $x,y\in \La$ and any linear map $F:\E_x \to \E_y$ 
\begin{equation}  \label{estMnorm}
K_\e (x) ^{-1} \| F \| \le \| F \|_{y\leftarrow x} \le K_\e (y) \| F \| \, .
\end{equation}
When $\e$ is fixed we will usually omit it and write $K(x)=K_\e (x)$ and 
$\|u\|_{x} =\|u\|_{x,\e}$.
\vskip.1cm

Similarly, we will consider the Lyapunov norm of a homogeneous 
polynomial map  $R:\E_x\to \E_{y}$ of order $n$ defined as
\begin{equation}\label{normDef}
 \|R\|_{y \leftarrow x }=\sup\,\{\,\|R(u)\|_{y,\e}:\; u\in \E_x, \;  \|u\|_{x,\e} =1  \,\}.
\end{equation}
 It follows
that
\begin{equation}\label{normP}
 \|\, R \circ P \,\|\leq \|R\|\cdot \|P\|^n.
\end{equation}
For a homogeneous polynomial map $P:\E_x\to \E_{y}$ of order $n$ we have
\begin{equation}\label {estPol}
 K_\e(x)^n \|P\| \le \|P\|_{y\leftarrow x} \le  K_\e(y) \|P \|.
 \end{equation}
This formula allows us to switch between the standard and Lyapunov norms
in spaces of polynomials and smooth functions.

\vskip.3cm


\section{Proof of Theorem \ref{NFext}} 
We give the proof for the case $\a>0$. The proof for $\a=0$ with
$N\ge2$ is similar but avoids difficulties of estimating the H\"older 
constant at $0$. 

We denote $\,\f_x^n = \f_{f^{n-1}x} \circ \dots \circ \f_{fx}\circ \f_x$.
We take $L=\max\, \{\kappa, M+N^3+2N^2 \}$, where $M=M(d)$ is
chosen to satisfy \eqref{Pn est}.
We set $\e_*=\e_0/ 3(N+1),$ where
\begin{equation}\label{epsilon}
\begin{aligned}
&\e_0=\min\,\{\, \nu/(2L+4(N+1+\a)),\; -\chi_\ell / (2L+2),\;-\lambda/(N^2+N+1)  \,\}>0,\\
&\text{where}\quad \nu=\chi_1-(N+\alpha)\chi_\ell>0 \quad\text{and }\,
\lambda<0 \,\text { is given by } \eqref{lambda}.
\end{aligned}
\end{equation}
We fix $\e<\e_0$ and  let $K=K_\e$ be as in \eqref{estLnorm}.
Since $\|\f_x \|_{\Cr}$ is $\e$-tempered, there is a function 
$C : \La \to [1,\infty)$ such that for all $x\in \La$ and $n\in \N$
\begin{equation} \label{estC}
\|\f_x \|_{\Cr} \le C(x)  \quad \text{and}\quad C(f^nx)\le e^{n\e} C(x).
\end{equation} 
Similarly, replacing $\sigma $ by a smaller function if necessary, 
we can assume that it satisfies
\begin{equation} \label{est sig}
\sigma : \La \to (0,1] \qquad  \text{and}\quad \sigma(f^nx)\ge e^{-n\e} \sigma(x).
\end{equation} 

\begin{lemma} \label{4.1}
Under the assumptions of Theorem \ref{NFext}, 
there exists a function $\rho  : \La \to  \R_+$ such that for all $x\in \La$, $n\in \N$, and $\,t\in B_{x, \rho(x)} \subset \E_x$, we have  $\rho (x) < \sigma (x) \le1$
and 
\begin{itemize}

\item[(1)] $\,\rho (f^nx)\ge e^{-\kappa\e n} \rho(x)$,  where $\kappa$ is given by \eqref{kappa},
\vskip.2cm

\item[(2)]  $\,\| D_t\,\f^n_x \|_{f^nx \leftarrow x} \le e^{(\chi_\ell+2\e)n}, $
 
\vskip.2cm

\item[(3)]  $\,\| D_t\,\f^n_x \| \le K(x)\, e^{(\chi_\ell+2\e)n}\, $,

\vskip.2cm
\item[(4)]  $\,\| \f_x^n(t) \| \le K(x)\, e^{(\chi_\ell+2\e)n} \|t\|$,

\vskip.2cm
\item[(5)]  $\,\| \f_x^n(t) \|_{f^nx} \le e^{(\chi_\ell+2\e)n} \|t\|_x$.
\end{itemize}
\end{lemma}

\begin{proof} 
We take $\b=1$ if $N\ge2$ and $\b = \a>0$ if $N=1$. For each $x\in \La$ we define 
\begin{equation}\label{rho}
\rho (x) = \sigma(x) [\e \, e^{\chi_\ell} (C(x)K(x)^2)^{-1}]^{1/\beta}.
\end{equation}
Then (1) follows
from \eqref{estC}, \eqref{est sig},  and \eqref{estK}; (5) follows from (2). We prove (2), (3), and  (4) by induction. The statements are clear for $n=0$, suppose they hold for $n$. Note that (2) implies (3) by \eqref{estMnorm}, and (3) implies (4). 
We observe  that 
$$
\| D_t\,\f^{n+1}_x \|_{f^{n+1}x \leftarrow x} \le  
\| D_{t'}\,\f_{f^{n}x} \|_{f^{n+1}x \leftarrow f^{n}x} \cdot 
\| D_t\,\f_{f^{n}x} \|_{f^nx \leftarrow x}, \quad\text{where }t'=\f_x^{n}(t).
$$
 Then (2) follows from the  inductive assumption and  
\begin{equation} \label{Dt'}
\| D_t\,\f_{f^{n}x} \|_{f^{n+1}x \leftarrow f^{n}x} \le e^{\chi_\ell+2\e}.
\end{equation} 
To prove \eqref{Dt'} we denote $\Delta=D_{t'}\,\f_{f^{n}x} -D_{0}\,\f_{f^{n}x}$.
By the choice of $\beta$, the $\beta$-H\"older constant of $D_{s}\,\f_{f^{n}x}$
at $0$ is at most $\,C(f^{n}x)$, so  using \eqref{estK}  we obtain
$$
\|\Delta \|_{f^{n+1}x \leftarrow f^{n}x} \le K(f^{n+1}x) \|\Delta \|\le K(f^{n+1}x)C(f^{n}x)\| t' \|^\beta \le 
$$
and using \eqref{estC} and the inductive assumption (4) we get
$$
\le K(x)\,C(x)\,e^{(2n+1)\e}  K(x) \,e^{\beta(\chi_\ell+2\e)n}\, \|t\|^\beta
 \le e^{\e} \,C(x)\,K(x)^2\, e^{[2\e+\beta(\chi_\ell+2\e)]n} \,\|t\|^\beta.
$$
 Since $\|t\| \le \rho (x)$ and $\beta\chi_\ell+2(1+\beta)\e\le 0$
we obtain 
$$\|\Delta \|_{f^{n+1}x \leftarrow f^{n}x} \le e^{\e}\, C(x)\,K(x)^2\, \rho(x)^\beta 
  \le  \e \, e^{\chi_\ell+\e} \sigma(x)^\beta \le \e \, e^{\chi_\ell+\e}.
$$\, Since 
$$
D_{0}\,\f_{f^{n}x}=F_{f^{n}x} \quad\text{and}\quad 
\| F_{f^{n}x} \|_{f^{n+1}x \leftarrow f^{n}x} \le e^{\chi_\ell+\e}
$$ 
by \eqref{estAnorm}, we conclude that
$$
\|D_{t'}\,\f_{f^{n}x}\|_{f^{n+1}x \leftarrow f^{n}x} \le 
\|\Delta \|_{f^{n+1}x \leftarrow f^{n}x}+ \|F_{f^{n}x}\|_{f^{n+1}x \leftarrow f^{n}x} \le \e \, e^{\chi_\ell+\e} +e^{\chi_\ell+\e} \le e^{\chi_\ell+2\e}.
$$ 
\end{proof}

\subsection{Construction of $\p$ and of the Taylor polynomial for  $\h$}$\;$

\noindent For each $x\in \La$ and map $\f_x : \E_{x} \to \E_{fx}$ we consider the 
Taylor polynomial at $t=0$: 
\begin{equation}\label{f_x}
\f_x(t) \sim \sum_{n=1}^N \fd^{(n)}_x(t).
\end{equation}
As a function of $t$, $\fd^{(n)}_x(t) :\E_{x} \to \E_{fx}$ is a  homogeneous polynomial map of degree $n$.
First we construct the Taylor polynomials at $t=0$ for the desired coordinate 
change $\h_x(t)$ and the polynomial extension $\p_x(t)$. We use  similar  notations for these Taylor polynomials:
$$\h_x(t)\sim \sum_{n=1}^N \hd^{(n)}_x(t)\quad \text{and} \quad 
 \p_x(t)= \sum_{n=1}^d \pd^{(n)}_x(t).
 $$
 For the first derivative we choose 
$$\hd^{(1)}_x= \Id : \E_x \to \E_x \quad \text{and} \quad 
\pd^{(1)} _x =\fd _x  \quad \text{for all } x \in \La.
$$

We construct the terms $\hd^{(n)}_x$ inductively to  ensure
that the terms $\pd^{(n)}_x$ determined by the conjugacy equation are 
of sub-resonance type. The base of the induction is the linear terms chosen above. For each $x \in \La$ we will construct $\hd^{(n)}_x$ and 
$\pd ^{(n)}_x$  which are measurable in $x$ and $n^2\e$-tempered, i.e. 
\begin{equation}\label{HPslow}
\sup_{k \in \N}\,  e^{-n^2 \e k} \, \| \hd^{(n)}_{f^kx}\|_{f^kx \leftarrow f^kx} < \infty 
\quad\text{and} \quad \sup_{k \in \N} \, e^{-n^2 \e k} \, \| \pd^{(n)}_{f^kx}\|_{f^kx \leftarrow f^kx} < \infty.
\end{equation}
Using these notations in the conjugacy equation 
$ \h_{fx}  \circ \f_{x} =\p_{x} \circ \h_{x} $ 
$$
\left( \Id   +\sum_{i=2}^N  H^{(i)} _{fx}  \right) \circ
\left( F_x+\sum_{i=2}^N F^{(i)} _{x} \right)  
\sim \left( F_{x} + \sum_{i=2}^d P^{(i)}_{x} \right) \circ
\left( \Id +\sum_{i=2}^N H^{(i)}_{x} \right).
$$
and considering the terms of degree $N\ge n \ge 2$,  we obtain  
$$
F^{(n)}_{x}\, + \, H^{(n)}_{fx} \circ F(x)\,+ \,
\sum H^{(i)}_{fx} \circ F^{(j)}_{x}
\,= \,F_{x} \circ H^{(n)}_{x}+ P^{(n)}_{x}+\,
\sum P^{(j)}_{x} \circ H^{(i)}_{x},
$$
where the summations are over all $i$ and $ j$ such that $i j=n$ and $ 1<i,j<n$.
We rewrite the equation as 
\begin{equation}\label{Pn}
F_{x}^{-1} \circ P^{(n)}_{x} =  - H^{(n)}_{x}  + F_{x}^{-1} \circ H^{(n)}_{fx} \circ F_{x} + Q_{x}, \;
\end{equation}
where
\begin{equation}\label{Q}
Q_{x}= F_{x}^{-1} \left( F^{(n)}_{x} + 
\sum_{ij=n, \;\, 1<i,j<n} H^{(i)}_{fx} \circ F^{(j)}_{x}  
-  P^{(j)}_{x} \circ H^{(i)}_{x} \right).
\end{equation}
We note that $Q_x$ is composed only of terms $H^{(i)}$ and  $P^{(i)}$ with $1<i<n$, which are already constructed, and terms $F^{(i)}$ with $1<i\le n$,
which are given. Thus by the inductive assumption $Q_x$ is defined for all
 $x \in \La$ and is measurable and tempered in $x$. 

Let $\po_x^{(n)}$ be  the space of all polynomial maps on  $\E_x$ of 
degree $n$, and  let $\s_x^{(n)}$ and $\n_x^{(n)}$ be the subspaces of 
sub-resonance and non sub-resonance  polynomials  respectively.
We seak $H^{(n)}_{x}$ so that the right side of \eqref{Pn}
 is in $\s_x^{(n)}$, and hence so is $P^{(n)}_{x}$ when defined by this equation.

Projecting  \eqref{Pn} to the factor bundle $\po^{(n)} / \s^{(n)}$, our goal is to solve the  equation
 \begin{equation}\label{barHn}
0 =  - \bar H^{(n)}_{x}  + F_{x}^{-1} \circ \bar H^{(n)}_{fx} \circ F_{x} + \bar Q_{x}, \;
\end{equation}
where $\bar H^{(n)}$ and $ \bar Q$ are the projections of $H^{(n)}$ and $Q$
respectively.

We consider the bundle automorphism $\Phi : \po^{(n)} \to \po^{(n)}$ covering 
$f^{-1}: \M \to \M$ given by
the maps $\Phi_x : \po^{(n)} _{fx} \to \po^{(n)} _{x}$
 \begin{equation}\label{Phi}
\Phi_x (R)= \fd _{x}^{-1} \circ R  \circ \fd _{x}.
 \end{equation}
Since $\fd$ preserves the splitting $\E=\E^1\oplus \dots \oplus \E^{\ell}$, it follows from the definition that the sub-bundles $\s^{(n)}$ and $\n^{(n)}$ are $\Phi $-invariant. 
We denote by $\bar \Phi$ the induced automorphism of $\po^{(n)} / \s^{(n)}$
and conclude that \eqref{barHn} is equivalent to 
 \begin{equation}\label{fixedP}
 \bar H^{(n)}_{x} =  \tilde \Phi_x ( \bar H^{(n)}_{fx} ), \quad \text{where }\; \tilde \Phi_x (R)= \bar \Phi_x (R) +  \bar Q_{x}.
\end{equation}
Thus a solution of \eqref{barHn} is a $\tilde \Phi$-invariant section of $\po^{(n)} / \s^{(n)}$. We will show that $\tilde \Phi$ is a nonuniform contraction 
and that it has a unique measurable tempered invariant section.
First, for polynomials of specific homogeneous type the exponent of $\Phi$ is 
determined  by the exponents of $F$ as follows.

\begin{lemma} \label{exponents}
For a polynomial $R: \E_{f x} \to \E_{f x}^i $ of homogeneous type  
$s=(s_1,  ..., s_{\ell})$ 
\begin{equation}
\| \Phi_x (R) \|_{x\leftarrow x} \le e^{ -\chi_i+\sum s_j \chi_j +(n+1)\e}\,\| R \|_{fx\leftarrow fx}.
\end{equation}
\end{lemma}


\begin{proof}
Suppose that $v=v_1+ \dots +v_{\ell}$, where $v_j \in \E^j_x$, and $\|v\|_x=1$. 
We denote $a_j=\| F|_{\E^j_x}\|_{fx\leftarrow x}$ and observe that
$F_x(v_j)= a_j v_j'  \in \E^j_{fx}$ with $\|v_j'\|_{fx}\le \|v_j\|_{x}$.
Since $R$ has homogeneous type $s= (s_1, \dots , s_\ell)$ we obtain 
by \eqref{stype} that
\begin{equation}
(R \circ F_x)(v)=R(a_1 v_1'+ \dots + a_\ell v_\ell ')= a_1^{s_1} \cdots  a_\ell^{s_\ell} \,\cdot R(v_1'+ \dots + v_\ell ').
\end{equation}
where $v'=v_1'+ \dots + v_\ell '$ has $\|v'\|_{fx}\le \|v \|_x=1$
by orthogonality of the splitting in the Lyapunov metric.
Thus 
$$
\| (R \circ F_x)(v)\|_{fx}= a_1^{s_1} \cdots  a_\ell^{s_\ell} \cdot
\| R(v')\|_{fx} \le a_1^{s_1} \cdots  a_\ell^{s_\ell} \,\cdot
\| R\|_{fx\leftarrow fx}
$$
for any $v \in \E_x$ with $\|v\|_x=1$, so we get 
$\| R \circ F_x\|_{x\leftarrow fx}\le a_1^{s_1} \cdots  a_\ell^{s_\ell} \,
\| R\|_{fx\leftarrow fx} $ by definition  \eqref{normDef}.
Now \eqref{normP} yields
$$
\| \Phi_x (R) \|_{x\leftarrow x} =\|F|_{\E^i_x}^{-1} \circ R \circ F_x\|_{x\leftarrow x} \le \|F|_{\E^i_x}^{-1}\|_{x\leftarrow fx} \cdot 
\| R \circ F_x\|_{x\leftarrow fx}\le
$$
$$
\le \|F|_{\E^i_x}^{-1}\|_{x\leftarrow fx} \cdot a_1^{s_1} \cdots  a_\ell^{s_\ell} \cdot
\| R\|_{fx\leftarrow fx} \le e^{-\chi_i +\e}  \cdot \prod_j (e^{\chi_j +\e})^{s_j}  \cdot  \| R \|_{fx\leftarrow fx} .
$$ 
Since $a_j=\| F|_{\E^j_x}\|_{fx\leftarrow x}\le e^{\chi_j +\e}$ and 
$\|F|_{\E^i_x}^{-1}\|_{x\leftarrow fx} \le e^{-\chi_i +\e}$ by \eqref{estAEi}.
\end{proof}
{\bf Remark.} Similarly, one can show that $\| \Phi^{-1} (R) \| \le   e^{\chi_i - \sum s_j \chi_j + (n+1)\e}$. Since this holds for any $\e>0$, one can obtain that
 the Lyapunov exponent of $\,\Phi$ on $R$ is  
$$\lim_{n \to \pm \infty} n^{-1} \log \| \Phi^n (R)\|=-\chi_i + \sum s_j \chi_j.$$

Taking the supremum of  $-\chi_i + \sum s_j \chi_j$ over all {\em non}
sub-resonance types $(i; s_1,  ..., s_{\ell})$, 
that is those for which this value  is negative, we define 
\begin{equation}\label{lambda}
\la =\sup \, \{ -\chi_i + \sum s_j \chi_j \}<0.
\end {equation} 
We note that  $\la <0$ since there only finitely many such values which are
greater than a given number. Thus we obtain the following lemma.

\begin{lemma} \label{contraction}
The map $\Phi : \n^{(n)} \to \n^{(n)}$ given by \eqref{Phi} is a nonuniform contraction over $f^{-1}$, and hence so is $\,\tilde \Phi : \po^{(n)} / \s^{(n)} \to \po^{(n)} / \s^{(n)}$  given by \eqref{fixedP}. More precisely,  
$\| \Phi_x (R) \|_{x} \le e^{\la +(n+1)\e}\| R \| _{fx}$.
\end{lemma}
\begin{proof}
The statement about $\,\tilde \Phi$ follows since the linear part 
$\bar \Phi$ of $\tilde \Phi$ is given by $\Phi$ when $\po^{(n)} / \s^{(n)}$ is naturally identified with $\n^{(n)}$.
\end{proof}

It follows from the previous remark that $\la$ is the maximal Lyapunov exponent   of $\Phi$ over $f^{-1}$ on the space of non sub-resonant polynomials, and that all Lyapunov exponents of $\Phi |_{\s^{(n)}}$  
are non-negative.

\vskip.1cm

Now we construct a $\tilde \Phi$-invariant  measurable section
of the bundle $\bb=\po^{(n)} / \s^{(n)}$ and study its properties. 
The construction is orbit-wise. We fix a point $x \in \La$, consider 
its positive orbit $\{ x_k=f^k x: \, k\ge 0\}$, and define the Banach space 
$$
\bx = \{ R=(R_k)_{k=0}^\infty: \,R_k\in \bb _{x_k}, \,
\|R\| < \infty \}, \; \text{ where }\,  \|R\|=\sup_{k\ge0}  e^{-\e n^2k} \| R_k \|_{x_k\leftarrow x_k} 
$$
and $\| R_k \|_{x_k\leftarrow x_k}$ is the norm induced on $\bb _{x_k}$ by the Lyapunov norm $\| . \|_{x_k}$ on $\E_x$. We denote $\qx = ( \bar Q_{x_k})_{k=0}^\infty$ and claim that it is in $\bx$. For this we need to estimate how 
the Lyapunov norm of \eqref{Q} grows along the trajectory: 
\begin{equation}\label{Q growth}
\begin{aligned}
& \|Q_{x_k}\|_{x_k\leftarrow x_k} \le  
\|F_{x_k}^{-1}\|_{x_{k}\leftarrow x_{k+1}} \cdot
( \,\|F^{(n)}_{x_k}\|_{x_{k+1}\leftarrow x_{k}} + \\
& \sum_{ij=n, \;\, 2\le i,j\le n/2} \|H^{(i)}_{x_{k+1}}\|_{x_{k+1}\leftarrow x_{k+1}}
 \|F^{(j)}_{x_k}  \|^i_{x_{k+1}\leftarrow x_{k}}
+  \|P^{(j)}_{x_k}\|_{x_{k+1}\leftarrow x_{k}}
 \|H^{(i)}_{x_k}\|^j_{x_{k}\leftarrow x_{k}}\, ).
 \end{aligned}
\end{equation}
First $\|F_{x_k}^{-1}\|_{x_{k}\leftarrow x_{k+1}}\le e^{-\chi_\ell+\e}$ for all $x$ and $k$ by \eqref{estAnorm}. The exponential growth rate in $k$ of $\|F^{(n)}_{x_k}\|_{x_{k+1}\leftarrow x_{k}}$ is at most $2\e$. Indeed, using \eqref{estPol} and 
\eqref{estK} we can obtain from \eqref{estC} the corresponding estimate 
for $\Cr$ norm with respect to the Lyapunov metric on $\E_{x_k}$: 
\begin{equation} \label{estCL}
\|\f_{x_k} \|_{\Cr,x_k} \le K(f{x_k}) \|\f_{x_k} \|_{\Cr} \le K({x_{k+1}})C({x_k})  \le e^{(2k+1)\e} K(x)C(x).
\end{equation} 
Then  using the inductive assumption 
\eqref{HPslow} we can estimate the exponential growth rate of the two terms
in the sum respectively as $(i^2+2i)\e$ and $(j^2+i^2j)\e$, which are at most
$((n/2)^2+in) \e < n^2\e$. So the exponential growth rate of 
$\|Q_{x_k}\|_{x_k\leftarrow x_k}$ can be estimated by $n^2\e$
and thus $\|\qx\|<\infty$.

Then $\tilde \Phi^x$ induces an operator
on $\bx$ by $
(\tilde \Phi^x (R)) _k =  \bar \Phi_{x_k} (R_{k+1})+ \qx _k
$ and we have
$$
\begin{aligned}
& \|\tilde \Phi^x (R) -\tilde \Phi^x (R')\|  =  \,
\sup_{k\ge0}  \, e^{-\e n^2k} \,\| \, \bar \Phi_{x_k} (R_{k+1}-R'_{k+1}) \,\|_{x_k\leftarrow x_k} \,\le \\
&\le \, \sup_{k\ge0} \, e^{-\e n^2k}  e^{\la +(n+1)\e} \, \| R_{k+1}-R'_{k+1} \|_{x_{k+1}\leftarrow x_{k+1}} \le \\
& \le   e^{\la +(n^2+n+1)\e } \sup_{k\ge0} \, e^{-\e n^2 (k+1)} \,\| (R_{k+1}-R'_{k+1}) \|_{x_{k+1}\leftarrow x_{k+1}}
\le e^{\la +(n^2+n+1)\e}\, \| R-R' \|.
\end{aligned}
$$
Since $\la +(n^2+n+1)\e<0$ by the choice of $\e$ \eqref{epsilon},   
$\tilde \Phi^x$ is a contraction
and thus has a unique fixed point $R^x\in \bx$. We claim that 
$\bar H^{(n)} _x = R^x_0$ is a measurable function which is a 
unique solution of \eqref{fixedP} or equivalently \eqref{barHn}. 
Measurability follows from the fact that the fixed point can be 
explicitly written as a series
\begin{equation}\label{bar s"}
\bar H^{(n)} _x  = \sum _{k=0}^\infty (F^k_{x})^{-1} \circ \bar Q_{x_k} \circ  F^k_{x}.
\end{equation}
Invariance is clear since $(R^x_{k+1})_{k=0}^\infty$ is a fixed point of
$\tilde \Phi^{fx}$ which coincides with  $(R^{fx}_{k})_{k=0}^\infty$ by 
uniqueness and thus $R^x_1=R^{fx}_0$. More generally, 
$\bar H^{(n)}_{x_k} =R^{x_k}_0 =R^x_k$, and since $R^x\in \bx$, the 
exponential growth rate of $\|\bar H^{(n)}_{x_k}\|_{x_k\leftarrow x_k}$ 
is at most $n^2\e$. Now we can choose $H^{(n)}_{x}$ as a lift  of $\bar H^{(n)} _x$
to $\po^{(n)}_x$ which is measurable in $x$ and satisfies \eqref{HPslow}.
Then we define $\pd^{(n)}_x$ by equation \eqref{Pn}. It also 
satisfies \eqref{HPslow} as so do $H$ and $Q$ and as 
$\|F_{x}\|_{x\leftarrow fx}$ and $\|F_{x}^{-1}\|_{fx\leftarrow x}$ are uniformly
bounded. This completes the inductive step.
\vskip.2cm
Thus we have constructed the Taylor polynomial for the coordinate change
\begin{equation}\label{H^N}
\h_x^N(t)= \sum_{n=1}^N \hd^{(n)}_x(t)\quad\text{of  order }\;
N\ge d=\lfloor \chi_1/\chi_{\ell} \rfloor
\end{equation}
 and the polynomial map 
$\p_x(t)= \sum_{n=1}^d \pd^{(n)}_x(t)$.

\vskip.4cm

\subsection{Construction of the coordinate change $\h$.} $\;$ 
\vskip.1cm
\noindent  We rewrite the conjugacy equation 
$\h_{fx} \circ \f_x=\p_x\circ \h_x$ in the form 
\begin{equation} \label{conj eq}
\h_x=\p_x^{-1} \circ \h_{fx} \circ \f_x.
\end{equation}
A solution $\h=\{ \h_x \}$ of this equation is a fixed point of the operator $T$
\begin{equation} \label{T}
T(\h)_x=\p_x^{-1} \circ \h_{fx} \circ \f_x.
\end{equation}
We denote $R=\h-\h^N$ and observe that $T(\h)=\h$ if and only if
\begin{equation} \label{Delta}
T(R)=R+\Delta^N, \;\,\text{ where}\quad\Delta^N = \h^N-T(\h^N).
\end{equation}
We will find $R$ as the fixed point of the 
 contraction  
\begin{equation} \label{tilde T} 
\tilde T(R)=T(R)+\Delta^N
\end{equation}
in an appropriate space of sequences of functions along an orbit.
By the construction of $\h^N$ and $\p$, $\,\h^N$ and $T(\h^N)$ have 
the same derivatives at the zero section up to order $N$. Hence 
$\Delta^N$ has vanishing  derivatives  at the zero section up to order $N$. 
To define the space we introduce some notations. For any $x \in \La$ we
denote by $B_{x,r}$ the ball  centered at $0$ in  $\E_x$ of 
radius $r<\rho(x)<1$ in the Lyapunov norm $\|.\|_x$. We define  
  $$
  \c_{x,r} = \{ R \in \Cr (B_{x,r},\E_x) : \;
  D^{(k)}_0 R =0, \; k=0,...,N \}.
  $$
In this proof we will consider the $\Cr$ norms {\em with respect to the 
Lyapunov metric} on $\E_{x}$. They are estimated through the norms for 
the standard metric \eqref{Crnorm} in \eqref{estCL}. In particular, we use
 the $\a$-H\"older constant \eqref{Canorm} of $D^{(N)} R$ at $0$ with 
 respect to the Lyapunov metric, which for any $R\in \c_{x,r}$ is given by
\begin{equation}\label{norm a}
 \|D^{(N)} R\|_{x,\a}= \sup \{ \| D^{(N)}_t R\|_{x\leftarrow x}\cdot \|t\|_x^{-\a} : \; 0\ne t \in B_{x,r}\}.
\end{equation} 
Also, for any $R\in \c_{x,r}$ we can estimate lower derivatives as 
\begin{equation}\label{deriv'}
 \|D^{(m)}_t R \|_{x\leftarrow x} \le \|t\|_x^{N-m} \, 
 \sup \,\{ \|D^{(N)}_s R \|_{x\leftarrow x} : \,\| s\|_x \le \|t\|_x \} ,
  \end{equation} 
so using the above H\"older constant we have that for any  $0 \le m<N$ 
and $t\in B_{x,r}$
\begin{equation}\label{deriv}
 \|D^{(m)}_t R \|_{x\leftarrow x}  \le \|t\|_x^{1+\a}\, \|D^{(N)} R\|_{x,\a}.
\end{equation} 
Thus the norms of all derivatives are dominated by the H\"older constant and hence
\begin{equation}\label{Cr=a}
  \|R \|_{\Cr (B_{x,r})} = \|D^{(N)} R \|_{x,\a}.
\end{equation} 
So we will take $\|D^{(N)} R \|_{x,\a}$ as the norm $\c_{x,r}$, which makes 
it into a Banach space.

Now we consider the corresponding spaces along the orbit $x_k=f^kx$.
We will specify later a large $L>1$ and a small $r=r(x)<\rho(x)$. Using
them we define $r_k=re^{-2Lk\e}$ and consider the  Banach space 
$$
\c^x = \{ \bar R=(R_k)_{k=0}^\infty: R_k\in \c_{x_k,r_k}, \,
\|\bar R\|_{\c^x} < \infty \}, \; \text{ where }  \|\bar R\|_{\c^x}=\sup_{k\ge0}  e^{-Lk\e} \|D^{(N)} R_k \|_{x_k,\a} 
$$
and the norm $\| . \|_{x_k,\a}$ is defined as in \eqref{norm a} and satisfies 
\eqref{Cr=a}. To ensure that $\bar \Delta^N = (\Delta^N_{x_k})$ is in $\c^x$
we need to estimate the growth of $\Cr$ norms 
of $\h^N_x$ and $T(\h^N)_x=\p_x^{-1} \circ \h^N_{fx} \circ \f_x$ along the orbit.

We recall that by the construction $D_{0}^{(1)} \, (\h_{x_k})= \Id$, and 
$D_{0}^{(1)} \, (\p_{x_k})= P^{(1)}_{x_k}=F_{x_k}$, which satisfy
$\| F_{x_k} \|_{x_{k+1}\leftarrow x_k} \le e^{\chi_\ell+\e}$ and
$\| F_{x_k}^{-1} \|_{x_k\leftarrow x_{k+1}} \le e^{\chi_1+\e}$. Also, 
for $2 \le n \le d$, Lyapunov norms of $D_{0}^{(n)} \, (\p_{x_k})= P^{(n)}_{x_k}$
and $D_{0}^{(n)} \, (\h_{x_k})= H^{(n)}_{x_k}$ grow at most at the 
exponential rate $n^2\e$ in $k$ by \eqref{HPslow}.

For $\h^N$, the derivative of order $N$ is constant $\hd^{(N)}_{x}$ on 
$\E_{x}$, and the lower derivatives on $B_{x,\rho(x)}$ can be 
inductively estimated by integration similarly to \eqref {deriv'}  
$$
 \|D^{(N-1)}_t \h_x \|_{x\leftarrow x}  \le 
 \|D^{(N-1)}_0 \h_x \|_{x\leftarrow x} + \|t\|_x \|\hd^{(N)}_{x}\|_{x\leftarrow x} 
 \le \|\hd^{(N-1)}_{x}\|_{x\leftarrow x} + \|\hd^{(N)}_{x}\|_{x\leftarrow x} 
$$
yielding the same estimate of the exponential rate as for $\hd^{(N)}_{x}$
\begin{equation}\label{H est}
 \|  \h_{x_k} \|_{\Cr(B\,{x_k,\rho(x_k)})} \le c_1(x) e^{N^2k\e}.
 \end{equation}

Since $\p _{x_k} ^{-1}$ is also a sub-resonance polynomial, its coefficients
can be obtained inductively from those of $\p _{x_k}$ and hence
there exists a constant $M=M(d)>d$  depending on $d$ only so that they grow 
at most at the exponential rate $M\e$ in $k$. The derivative of order $d$
is constant on $\E_{x_k}$, higher derivatives are zero, and the lower 
derivatives can be estimated  as for $\h$, so we obtain 
 for all $k \ge 0$
\begin{equation}\label{Pn est}
 \| (\p_{x_k})^{-1} \|_{\Cr(B\,{x_k,\rho(x_k)})} \le c_2(x) e^{Mk\e}.
 \end{equation}

To obtain estimates for $(T(\h^N))_x=\p_x^{-1} \circ \h^N_{fx} \circ \f_x$ we
use the following lemma.

\begin{lemma} \label{QF}
If $Q$ is a polynomial of degree at most $N$ and $\f $ is $\Cr$ 
then $Q \circ \f$ is $\Cr$ and 
$\, \|Q \circ \f \|_{\Cr} \le c_N\, \|Q\|_{C^N} \,  \| \f \|_{\Cr}^N$ where $c_N$
depends on $N$ only.
\end{lemma}
\begin{proof}
Since $Q$ is $C^\infty$ it is clear that $Q \circ \f$ is $C^N$. 
For the $N^{th}$ derivative we have
$$
   D^{(N)}_t \,(Q \circ \f)=  D_{\f(t)} Q \circ D^{(N)}_t \f +
   \sum_{kj=N, \; j<N} D^{(k)}_{\f(t)} Q \circ D^{(j)}_t \f .
 $$
First we estimate $\a$-H\"older constant at $0$ of the first term. 
As $DQ$ is linear, we get
$$
   D_{\f(t)} Q \circ D^{(N)}_t \f - D_{0} Q \circ D^{(N)}_0 \f =
   (D_{\f(t)} Q- D_{0} Q ) \circ D^{(N)}_t \f  +  
   D_{0} Q \circ ( D^{(N)}_t \f - D^{(N)}_0 \f )
 $$
whose norm can be estimated by
$$
   \| D_{\f(t)} Q - D_{0} Q \| \cdot \|D^{(N)}_t \f \| +  
   \|D_{0} Q \| \cdot \| D^{(N)}_t \f - D^{(N)}_0 \f \| \le
 $$
 $$
 \|Q\|_{C^2} \cdot  \|\f(t) \| \cdot  \| \f \|_{\Cr} + \, 
   \|Q\|_{C^1} \cdot \| \f \|_{\Cr}\cdot  \|t \|^\a \le 
    $$
 $$
   \|Q\|_{C^2} \cdot   \| \f \|_{\Cr} \cdot \|\f\|_{C^1} \cdot  \|t \| + \, 
   \|Q\|_{C^1} \cdot  \| \f \|_{\Cr}\cdot  \|t \|^\a.
 $$
So the $\a$-H\"older constant at $0$ of $D_{\f(t)} Q \circ D^{(N)}_t \f$ 
is estimated by $2\|Q\|_{C^N} \,  \| \f \|_{\Cr}^2$.
The other terms in the sum are $C^1$ and hence are Lipschitz with constant 
bounded by supremum norms of their derivatives. These norms, along with 
the norms of lower derivatives of $Q \circ \f$ can be estimated as a sum
of termss of the type
$$
  \| D^{(k)}_{\f(t)} Q \circ D^{(j)}_t \f )\| \le   
  \| D^{(k)}_{\f(t)} Q \| \cdot \|D^{(j)}_t \f \|^k \le \|Q\|_{C^N} \,  \| \f \|_{\Cr}^N.
 $$
We conclude that $ \|Q \circ \f \|_{\Cr} \le c_N\, \|Q\|_{C^N} \,  \| \f \|_{\Cr}^N$.
\end{proof}

We apply the lemma with $Q=\h^N$ and then with $Q=\p_x^{-1}$.
We conclude that $T(\h^N)$ is $\Cr$. Moreover, since $\| \f \|_{\Cr}^N$
is $2\e$-tempered by \eqref{estCL} and using \eqref{Pn est} and \eqref{H est}
we obtain that for $M'=M+N^3+2N^2 $
\begin{equation}\label{TH est}
 \| T(\h^N) \|_{\Cr(B\,{x_k,\rho(x_k)})} \le c_3(x) e^{M'k\e}.
 \end{equation}
Recall that we chose $L\ge \max \{ \kappa, M' \}$ and $r<\rho (x)$.
Then we obtain by the definition of $r_k$ and Lemma \ref{4.1} (1) 
that for all $k \ge 0$
\begin{equation}\label{r est}
  r_k =r e^{-2L\e k} < e^{-L\e k} \rho(x_k) \le  \rho(x_k).
 \end{equation}
We conclude that with such choices we have $\bar \Delta^N\in \c^x$ with
 $$
\|\bar \Delta^N\|_{\c^x}\le D'= \sup_{k\ge0}  e^{-Lk\e} \| \Delta _k \|_{\Cr(B\,{x_k,\rho(x_k)})}< \infty.
 $$

Now we consider the operator induced by $T$  on $\c^x$:
$$
(T^x (\bar R)) _k =   (\p_{x_k})^{-1} \circ R_{k+1} \circ  \f_{x_k} .
$$
We denote by $B^x(D)$ the ball of radius $D=D'/\theta$ in $\c^x$,
where $\theta>0$ given by \eqref{theta}. We will choose $L$ and $r$ 
so that for any
$\bar R \in B^x(D)$ the maps $(T (R)) _k$ are defined on 
$B_{x_k,r_k}$ and $\|T^x (\bar R)\|_{\c^x} \le (1-\theta) \| \bar R\|_{\c^x}$.
Then it will follow that   $\tilde T^x  : B^x(D) \to B^x(D)$ and is also a 
contraction, whose unique fixed point gives the desired solution.

First we check that the compositions in $(T (R)) _k$ are well-defined. 
We take $t \in B_{x_k,r_k}$ and show that $t'= \f_{x_k}(t)$ is in 
$B_{x_{k+1},r_{k+1}}$. Since
by \eqref{r est} $t$ is in the ball $B_{x_k,\rho(x_k)}$ in standard metric,
the estimates in Lemma \ref{4.1} hold for any $k$.
In particular, by  (2),(5)
\begin{equation}\label{t'}
\|  D_t^{(1)} \f_{x_k} \|_{x_{k+1}\leftarrow x_k} \le e^{\chi_\ell+2\e} \quad \text{and}
\quad \|\f_{x_k}(t)\|_{x_{k+1}}\le e^{\chi_\ell+2\e} \|t\|_{x_k} < \|t\|_{x_k},
 \end{equation}
the last since $\chi_\ell+2\e<0$, which also yields 
\begin{equation}\label{t'2}
 \|t'\|_{x_{k+1}}= \|\f_{x_k}(t)\|_{x_{k+1}}\le e^{\chi_\ell+2\e}re^{-2Lk\e}
\le re^{-2L(k+1)\e} =r_{k+1}, 
 \end{equation}
since by the choice of $\e$ we have
\begin{equation}\label{echi_l}
\chi_\ell+2L\e+2\e\le0. 
 \end{equation}

Estimating $t''=R_{k+1}(t')$ using \eqref{deriv'} and \eqref{Cr=a} we obtain that
for any $R \in B^x(D)$
\begin{equation}\label{t''}
\|t''\|_{x_{k+1}} 
\le \|t'\|_x \|D^{(N)} R_{k+1}\|_{x_{k+1},\a}
 \le r_{k+1} e^{L(k+1)\e} \|\bar R\|_{\c^x} \le r e^{-L(k+1)\e} D < \rho(x_{k+1})
\end{equation} 
by \eqref{r est}, provided that $rD<\rho(x)$.

Now we will show that $T^x$ is a contraction on $B^x(D)$. For this we first estimate $\|D^{(N)}  (T(\bar R))_k \|_{x_k,\a}$. We consider
 \begin{equation} \label{est}
\begin{aligned}
  & D_t^{(N)} \, (T^x(R))_k  = D_t^{(N)}  \left( (\p_{x_k})^{-1} \circ R_{k+1} \circ  \f_{x_k} \right)= \\
  & = D_{t''}^{(1)} \, (\p_{x_k})^{-1} \circ D_{t'}^{(N)} R_{k+1} \circ D_t^{(1)} \f_{x_k} \,+ \,J,
  \end{aligned}
\end{equation}
where we again denoted $t'= \f_{x_k}(t)$ and $t''=R_{k+1}(t')$, and where $J$ consists of  a fixed number of terms of the type
$$
   D_{t''}^{(i)} \, (\p_{x_k})^{-1} \circ D_{t'}^{(j)} R_{k+1} \circ D_t^{(m)} \f_{x_k}, \quad ijm=N, \; j<N.
  $$
Their norm can be estimated using \eqref{normP} as 
$$
  \| D_{t''}^{(i)} \, (\p_{x_k})^{-1} \|_{x_k\leftarrow x_{k+1}} \cdot \| D_{t'}^{(j)} R_{k+1} \|^i_{x_{k+1}\leftarrow x_{k+1}}\cdot \| D_t^{(m)} \f_{x_k}\|^{ij}_{x_{k+1}\leftarrow x_{k}}.
  $$
For the last term we have 
$\| D_t^{(m)} \f_{x_k}\|_{x_{k+1}\leftarrow x_{k}} \le K(x)C(x) e^{(2k+1)\e} $ 
by \eqref{estCL}. For the middle term we have by \eqref{deriv} and \eqref{t'} 
as $j\le N-1$ 
 $$
 \| D_{t'}^{(j)} R_{k+1}\|_{x_{k+1}\leftarrow x_{k+1}} \le 
  \|t'\|_{x_{k+1}}^{1+\a} \cdot \| D^{(N)} R_{k+1} \|_{x_{k+1},\a} 
 < \|t\|_{x_k}^{1+\a} \cdot \|\bar R \|_{\c^x}\,e^{L(k+1)\e}. 
 $$ 
 For the first term, since $t'' \in B_{x_{k+1},\rho(x_{k+1})}$ by \eqref{t''},
 we use \eqref{Pn est} to get
\begin{equation}\label{Pn est'}
 \| D_{t''}^{(i)} \, (\p_{x_k})^{-1} \|_{x_k\leftarrow x_{k+1}} \le c_2(x) e^{Mk\e}, 
 \end{equation}
for all $k$ and $1 \le i \le d$. Therefore, with $M''=M+2+L$
\begin{equation}\label{J}
\| J\| < c_4(x)  e^{(Mk+L(k+1)+2k+1)\e}\, \|t\|_{x_k}^{1+\a} \, \|\bar R \|_{\c^x}
< c_4(x) e^{M''\e (k+1)}  \, \|\bar R \|_{\c^x} \,r_k \,\|t\|_{x_k}^{\a}.
\end{equation}

Now we estimate  the main term in \eqref{est}.  As we observed before,
estimates in Lemma \ref{4.1} apply to $t, t', t''$. In particular, we use \eqref{t'} for $\f$ term. For the $\p$ term we claim that
\begin{equation}\label{Pt''}
\| D_{t''}^{(1)} \, (\p_{x_k})^{-1} \|_{x_k\leftarrow x_{k+1}} \le e^{\chi_1+2\e}.
\end{equation}
This follows  from 
$$\| D_{0}^{(1)} \, (\p_{x_k})^{-1} \|_{x_k\leftarrow x_{k+1}} =
\|  \, \fd_{x_k}^{-1} \|_{x_k\leftarrow x_{k+1}}\le e^{\chi_1+\e}$$
similarly to \eqref{Dt'} in Lemma \ref{4.1}. Indeed, if $d=1$ then this follows 
as $D^{(1)} \, (\p_{x_k})$ is constant. If $N \ge 2$ then the Lipschitz  constant 
of $D^{(1)} \, (\p_{x_k})^{-1}$ is at most $c_2(x) e^{Mk\e}$  by \eqref{Pn est}, 
so  using \eqref{t''} we obtain as $M<L$
$$
\|D_{t''}^{(1)} \, (\p_{x_k})^{-1} -D_{0}^{(1)} \, (\p_{x_k})^{-1} \|_{x_k\leftarrow x_{k+1}} \le c_2(x) e^{Mk\e} \| t'' \|_{x_{k+1}}  \le  c_2(x) rD e^{(M-L)(k+1)\e} < \e \, e^{\chi_1}
$$
 provided that $r<e^{\chi_1}(c_2(x)D)^{-1}$.

 We conclude using \eqref{Pt''}, \eqref{t'}, and \eqref{t'2} that 
\begin{equation} \label{main term}
\begin{aligned}
&\| D_{t''}^{(1)} \, (\p_{x_k})^{-1} \circ D_{t'}^{(N)} R_{k+1} \circ D_t^{(1)} \f_{x_k} \|_{x_k\leftarrow x_{k}} \,\le \\
&\le \,
  \| D_{t''}^{(1)} (\p_{x_k})^{-1} \|_{x_{k}\leftarrow x_{k+1}} \cdot \|D^{(N)} R_{k+1} \|_{x_{k+1},\a}\, \|t'\|_{x_{k+1}} ^\a \cdot \|  D_t^{(1)} \f_{x_k} \|_{x_{k+1}\leftarrow x_{k}}  ^N \,\le \\
 & \le \, e^ {-\chi_1+2\e } \cdot \|\bar R \|_{\c^x} e^{L(k+1)\e} e^{\a (\chi_\ell+2\e)}
 \|t\|_{x_k}^\a   \cdot e^{N (\chi_\ell+2\e)}  \,= \, e^{-\nu+L'\e} \|t\|_{x_k}^\a\, \|\bar R \|_{\c^x}\, e^{Lk\e}  ,
 \end{aligned}
\end{equation}
where $\nu=-(N+\a)\chi_\ell+\chi_1>0$ and $L'=2+L+2(N+\a)$. 
Provided that
\begin{equation}\label{e_0 nu}
\e \le \e_0 \le \nu/(2L') 
 \end{equation}
we obtain that  $e^{-\nu+L'\e} \le e^{-\nu/2}= 1-2\theta$ where we defined
\begin{equation}\label{theta}
 \theta = (1- e^{-\nu/2})/2 >0.
 \end{equation}
Combining the estimates \eqref{J} and \eqref{main term} we get 
for $\bar R \in B^x(D)$
$$
\|D_t^{(N)} \, (T^x(\bar R))_k \|_{x_k\leftarrow x_{k}} \le \, \|t\|_{x_k}^\a\, \|\bar R \|_{\c^x}\, e^{Lk\e}\, [1-2\theta + c_4(x) r_k  e^{\e (M'' (k+1)-Lk)} ] .
$$
Since $r_k=r e^{-2Lk\e}$ and $3L>M''$ we see that for all $k\ge 0$
$$
c_4(x)\, r_k\,  e^{\e (M'' (k+1)-Lk)} \le c_4(x)\, r  e^{\e (M'' (k+1)-3Lk)} \le
  c_4(x)\, r e^{\e M''} \le \theta 
$$
if we choose $r$ satisfying $r\le  \theta /(c_4(x) e^{\e M''} )$ in addition to 
$r<\rho (x)/D= \theta \rho (x)/D'$.
Then for all $\bar R \in B^x(D)$ we obtain
$$
\|D^{(N)} \, (T^x(\bar R))_k \|_{x_k,\a} \le \,  (1-\theta) \|\bar R \|_{\c^x}e^{Lk\e}
\quad \text{and}
$$
$$
\|T^x (\bar R)\|_{\c^x} =  \sup_k \, e^{-Lk\e} \|D^{(N)}  (T^x(\bar R))_k \|_{x_k,\a} 
\le \, (1-\theta) \|\bar R \|_{\c^x} .
$$

Since $\| \bar \Delta^N\|_{\c^x} \le D'= \theta D$ the operator 
$\tilde T^x(\bar R)=T(\bar R)+\bar \Delta^N$ is also a contraction and 
preserves $B^x(D)$. Thus $\tilde T^x$ has a unique fixed point 
$\bar R^x \in B^x(D)$ which depends measurably on $x$. As in the 
construction of Taylor coefficients, the uniqueness implies that 
$(R^x)_0$ is $L\e$-tempered and solves the equations \eqref{tilde T} 
and \eqref {Delta}. Thus
the measurable family of coordinate changes $ \h_x=\h^N_x+(R^x)_0$, 
 is also $L\e$-tempered and  conjugates $\p_x$ and $\f_x$.

We conclude that the maps $\h_x$ is a family of $\Cr$ diffeomorphisms 
defined on $B_{x,r(x)}$ which depend measurably on $x\in \M$  and $L\e$-tempered. Since $\chi_1 +2\e +L\e <0$, we can extend each $\h_x$ to 
$B_{x,\rho(x)}$.
Indeed, by Lemma \ref{4.1}  for each $t \in B_{x,\rho(x)}$ we will have 
$\f^k_x (t) \in B_{x_k,r_k}$ for some $k$. Then $\h_x$ is defined uniquely
by invariance.

\vskip.5cm

\subsection {Prove of part (2): ``uniqueness" of $\h$}
This essentially follows from the ``uniqueness" of the construction.
First we construct inductively coordinate changes 
$ \h_k=\{ \h_{k,x}\}$ for $k=1,..., N$ with $ \h_1 =\tilde \h$. 
Consider the Taylor series
$$\h_{1,x}(t)= \sum_{n=1}^\infty \hd^{(n)}_{1,x}(t). $$
By assumption $\hd^{(1)}_{1,x}=\hd^{(1)}_{x}=\Id$. Then $\hd^{(2)}_{1,x}$
and $\hd^{(2)}_{x}$ satisfy the same equation \eqref{barHn} when projected
to the factor-bundle $\po^{(2)} / \s^{(2)}$. By uniqueness of the solution of 
\eqref{barHn} we obtain that $\hd^{(2)}_{x}=\hd^{(2)}_{1,x} + \Delta^{(2)}_x$,
where $\Delta^{(2)}_x \in \s^{(2)}_x$, and then the polynomial $\Id + \Delta^{(2)}_x$
is in $G_\chi$. Now we consider the coordinate change 
$\h_{2,x}=(\Id + \Delta^{(2)}_x)\circ \h_{1,x}$, which 
 conjugates $\f$ to a new normal form 
 $$\p_{2,x}= (\Id + \Delta^{(2)}_{fx})\circ\p_{1,x} \circ (\Id + \Delta^{(2)}_x)^{-1}
 $$
 which is also of sub-resonance type.
By construction $\hd^{(2)}_{2,x}=\hd^{(2)}_{1,x} + \Delta^{(2)}_x=\hd^{(2)}_{x}$,
so that $\h$ and $\h_2$ have the same Taylor terms up to order two.
 
\vskip.1cm
Inductively, suppose $\h_{k-1}$ is constricted so that 
$$
 \hd^{(n)}_{k-1,x}
\text{ are }n^2\e\text{-tempered for }n=1,...,N, \quad
\hd^{(n)}_{x}=\hd^{(n)}_{k-1,x} \,\text{ for }n=1,...,k-1,
$$ 
and the corresponding normal form  $\p_{k-1 ,x}$ is of 
sub-resonance type. It follows that $\p$ and $\p_{k-1}$  have the same 
terms up to order $k-1$. Hence $\hd^{(k)}_{k-1,x}$
and $\hd^{(k)}_{x}$ satisfy the same equation \eqref{barHn} when projected
to the factor-bundle $\po^{(k)} / \s^{(k)}$. Indeed, the $Q$ term defined by 
\eqref{Q} is composed only of $F^{(i)}$ and terms $H^{(i)}$ and  $P^{(i)}$ with $1<i<k$, which are the same for $\h_{k-1}$ and $\h$.
By uniqueness we obtain that 
$$
\hd^{(k)}_{x}=\hd^{(k)}_{k-1,x} + \Delta^{(k)}_x,\,\text{ where }\,
\Delta^{(k)}_x \in \s^{(k)}_x.
$$
 Then the coordinate change $\h_{k,x}=(\Id + \Delta^{(k)}_x)\circ \h_{k-1,x}$ has the same Taylor terms as $\h$ up to order $k$ and, since the polynomial $\Id + \Delta^{(k)}_x$ is in $G_\chi$, $\h_{k}$ conjugates $\f$ to a sub-resonance normal form 
$\p_{k,x}= (\Id + \Delta^{(k)}_{fx})\circ\p_{k-1,x} \circ (\Id + \Delta^{(k)}_x)^{-1}$. 
To complete the inductive step we need to show that $\| \hd^{(n)}_{k,x} \|$
is $n^2\e$-tempered. It suffices to show this for $\|R^{(n)}\|$ where
$R=\Delta^{(k)}_x\circ \h_{k-1,x}$. Since $\Delta^{(k)}_x$ is homogeneous 
of degree $k$, we have for $j=n/k$
$$\|R^{(n)}\|=\|\Delta^{(k)}_x\circ \h^{(j)}_{k-1,x}\| \le 
\|\Delta^{(k)}_x\| \cdot \| \h^{(j)}_{k-1,x}\|^k,$$
which is $(k^2+j^2k)\e$-tempered by the inductive assumption and 
the definition of $\Delta^{(k)}_x$. Since $j\le n/2$ as $k\ge2$ we get 
$j^2k=jn \le n^2/2$.
If also $j\ge2$ we have $k^2\le n^2/4$ and we obtain $n^2\e$-temperedness.
If $j=1$, we have $R^{(k)}=\Delta^{(k)}$, which is also $k^2\e$-tempered.

Thus in $N$ steps we obtain the coordinate change   
$$
\h_{N,x}= G_x \circ \tilde \h_{x}, \;\text{ where }\,
G_x = (\Id + \Delta^{(N)}_x) \circ \dots \circ (\Id + \Delta^{(2)}_x) \in G_\chi,
$$
 which has the same Taylor terms at $0$ as $\h$ up to order $N$.
 In fact,  for $n>d$ we have $\s^{(n)}=0$ and hence $\Delta^{(n)}=0$,
so that $\h_{N}=\h_{d}$. 
Now we show that $\h=\h_{N}$, which also proves the last statement in 
part (2) of the theorem. The equality follows from the uniqueness in 
the final step of the construction.  Indeed the difference $R=\h-\h_N$ 
has zero derivatives up to order $N$ at the zero section and
satisfies $R=T(R)$ for the operator $T$ from \eqref{T}. Hence
$R=0$ by uniqueness of the fixed point in the appropriate space 
$\c_{r,x}$ on which $T$ induces a contraction. To ensure that
the sequence $(R_{x_k})$ is in $\c_{r,x}$ we need estimate 
temperedness of $\a$-H\"older constant at $0$ for $\h^{(N)}_{N}$.
As above one can see that all terms in the polynomial $G_x$
are $N^2\e$-tempered. Then using Lemma \ref{QF} and the assumption 
on $\tilde \h$ we obtain that $\| \h_{N,x}\|_{\Cr}$ is $\tilde L\e $-tempered
for $\tilde L= (N^2+NL) < (N+1)L$. Hence $(R_{x_k})$ is in $\c_{r,x}$ with
$\tilde L$ in place of $L$, on which $T$ is a contraction if $\e <\e_1= \e_0/(N+1)$, 
as the inequalities \eqref{e_0 nu} and \eqref{echi_l} are satisfied. 
Thus $(R_{x_k})=0$ and extending by invariance, as \eqref{echi_l} is satisfied,
we conclude that $R_x=0$ on $B_{x,\rho(x)}$, and so $\h=\h_{N}$.

\subsection {Proof of Corollary \ref{Cinf}}

By part (2) of Theorem \ref{NFext}, if we fix a choice of Taylor polynomials 
of degree $d$ for $\h_x$, then the family $\h_x$ is unique. Then for each 
$N>d$ we can do the construction in part (1) with this fixed choice of 
Taylor polynomials and obtain the family of $C^{N}$ diffeomorphisms
$\h_x$. By uniqueness, all these families coincide and hence
 $\h_x$ are $\Ci$ diffeomorphisms. 


\subsection {Proof of part (3): Centralizer of $\h$}

First we prove that the derivative at zero section $\G_x=D_0\g_x$ is sub-resonance.
This is equivalent to the fact that $\G_x$ preserves the fast flag associated with
the Lyapunov splitting. 
Suppose to the contrary that for some $x \in \La$ and some $i<j$ we have a vector
$t$ in $E^i_x$ such that $t'=\G_x(t)$ has nonzero component $t'_j$ in $E^j_{gx}$.
Then 
$$
\| (F ^n_{gx}  \circ \G_{x} )(t) \|_{f^ngx} \ge \| F ^n_{gx} (t'_j) \|_{f^ngx}
 \ge e^{(\chi_j -\e)n}\, \| t'_j \|_{gx}
$$
and on the other hand
$$
\| (F ^n_{gx}  \circ \G_{x} )(t) \|_{f^ngx} = \| \G_{f^n x} ( F ^n_x \,t) \|_{gf^nx} \le 
\| \G_{f^n x} \|_{gf^nx\leftarrow f^nx} \cdot e^{(\chi_i +\e)n} \| t \|_x \le C e^{(\chi_i +3\e)n}
$$
which is impossible for large $n$ since $\e$ is small. Here we used that
the $\Cr$ norm $\|\g_{x} \|_{\Cr,x}$  with respect to the Lyapunov metric on $\E_{x}$  is $2\e$-tempered. This follows as in \eqref{estCL} since $\|\g_{x} \|_{\Cr}$ in standard norm is $\e$-tempered by assumption.

We conclude that $\G_x$ is sub-resonance for each $x\in \La$. 
Now we consider a new family of coordinate changes
$$\tilde \h_x = \G_x^{-1}\circ \h_{gx} \circ \g_x$$ 
which also satisfies $\tilde \h_x(0)=0$ and $D_0\tilde \h_x=\Id$. A direct calculation shows that
$$
\begin{aligned}
\tilde \h_{fx} \circ \f_x \circ \tilde \h_x & \,=\,  
\G_{fx}^{-1}\circ \h_{fgx} \circ \g_{fx} \circ \f_x \circ  \g_x^{-1}\circ \h_{gx}^{-1} \circ \G_x = \\
&\, = \,\G_{fx}^{-1}\circ \h_{fgx} \circ  \f_{gx} \circ \h_{gx}^{-1} \circ \G_x \,= \,
\G_{fx}^{-1}\circ \p_{gx} \circ \G_x \,= \,
\tilde \p_x,
\end{aligned}
$$
where $\tilde \p_x $ is a sub-resonance
polynomial as a product of sub-resonance polynomials. Now we would like to  to 
apply the uniqueness part of the theorem, which would give
 $\tilde \h_x = G_x \h_x $ for some 
tempered function $G_x \in G_\chi$. Then it follows from the definition of 
$\tilde \h_x$ that
$$  \h_{gx} \circ \g_x = \G_x \circ \tilde \h_x = (\G_x G_x) \circ \h_x $$
so that $\h_{gx} \circ \g_x \circ \h_x^{-1}=\G_x G_x $, which is again
a sub-resonance polynomial, as claimed.

To complete the proof it remains to show that $\tilde \h_x$ is suitably tempered to obtain 
uniqueness. 
As before, we can estimate the Lyapunov norm of the $n^{th}$ Taylor 
term of $\tilde \h_x$ as
$\|\tilde \h_x^{(n)}\|=\| \G_x^{-1}\| \circ \| \h_{gx}^{(k)}\| \circ \| \g_x^{(j)}\|^k $
with $n=kj$ and obtain that it is $m\e$-tempered with $m \le 2+k^2+ 2k <3 n^2$ for $n \ge2$.  Since $\| \h \|_{\Cr}$ is $L\e$-tempered, using 
Lemma \ref{GF} below with $Q=\h$ and $\f =\g$ we obtain that  
$\|  \h \circ \g \|_{\Cr}$ is $(L+ 2(N+\a))\e$-tempered.  Then  Lemma \ref{QF} implies that 
$\| \tilde \h \|_{\Cr}$ is $3L\e$-tempered as $(2+ L+ 2(N+\a))\le 3L$
(provided that $L\ge N+2$).
So the uniqueness result applies if $\e <\e_*=\e_1/3=\e_0/3(N+1)$.

\begin{lemma} \label{GF}
If $Q$ and $\f $ are $\Cr$, 
then $Q \circ \f$ is $\Cr$ and \\
$ \|Q \circ \f \|_{\Cr} \le c_N'' \, \|Q\|_{\Cr} \,  \| \f \|_{\Cr}^{N+\a},\,$ where $c_N''$
depends on $N$ only.
\end{lemma}
\begin{proof}
The proof is the same as in Lemma \ref{QF} except that, since $D^{(N)} Q$ is only H\"older, we also need to estimate the $\a$-H\"older constant at $0$ 
of the  term $D^{(N)}_{\f(t)} Q \circ D_t \f$ in
$$
   D^{(N)}_t \,(Q \circ \f)=  D^{(N)}_{\f(t)} Q \circ D_t \f + D_{\f(t)} Q \circ D^{(N)}_t \f + 
   \sum_{kj=N, \; j,k<N} D^{(k)}_{\f(t)} Q \circ D^{(j)}_t \f .
 $$ 
$$
\begin{aligned}
   D^{(N)}_{\f(t)} Q \circ D_t \f & -  D_{0}^{(N)} Q \circ D_0 \f = \\
   & = (D^{(N)}_{\f(t)} Q- D^{(N)}_{0} Q ) \circ D_t \f  +  
   D^{(N)}_{0} Q \circ D^{(N)}_t \f - D^{(N)}_{0} Q \circ D^{(N)}_0 \f 
   \end{aligned}
 $$
and its norm can be estimated by
$$
   \|  Q \|_{\Cr} \| \f(t)\|^\a \cdot \|D_t \f \|^N +  
   Lip(D^{(N)}_{0} Q)  \cdot \| D_t \f - D_0 \f \| \le
 $$
 $$
 \|Q\|_{\Cr} \cdot (\|\f\|_{C^1} \|t \|)^\a  \cdot  \| \f \|_{C^1}^N + \, 
  c'_N\,  \|D^{(N)}_{0} Q\| \, \| \f \|_{C^1}^{N-1} \cdot \| \f \|_{C^{1,\a}}\cdot  \|t \|^\a \le 
    $$
 $$
  \|t \|^\a \, (  \|Q\|_{\Cr} \cdot \|\f\|_{C^1}^{N+\a}  + \, 
   c'_N \, \|Q\|_{C^N} \cdot \| \f \|_{C^1}^{N-1} \cdot \| \f \|_{C^{1,\a}})  .
 $$
Here we estimated the Lipschitz constant $Lip(D^{(N)}_{0} Q)$ of
the homogeneous polynomial $N$-form $D^{(N)}_{0} Q$ on a ball 
of radius $R=\| \f \|_{C^1}$ by the supremum of its derivative on that 
ball, which is a homogeneous polynomial $(N-1)$-form whose norm 
can be estimated by $\|D^{(N)}_{0} Q\|$ with some constant $c_N'$ 
depending on $N$ only.

So the $\a$-H\"older constant at $0$ of $D^{(N)}_{\f(t)} Q \circ D_t \f$ 
is estimated by 
$$\|Q\|_{\Cr} (\|\f\|_{C^1}^{N+\a} +c'_N \| \f \|_{C^{1,\a}}^N)\le 
(c'_N+1) \|Q\|_{\Cr} \|\f\|_{\Cr}^{N+\a}.
$$
We conclude as in Lemma \ref{QF} that $ \|Q \circ \f \|_{\Cr} \le c_N'' \, \|Q\|_{\Cr} \,  \| \f \|_{\Cr}^{N+\a}$.

\end{proof}

This completes the proof of Theorem \ref{NFext}.
$\QED$


\section{Proof of Theorem \ref{NFfol} } 

\subsection{Proof of (i), (ii), (iii), (v).}

We will apply Theorem \ref{NFext}. First we note that the integrability condition 
for the derivative in Theorem \ref{NFext} was used in the proof only to obtain the 
Lyapunov splitting and the Lyapunov metric. So while the restriction $Df|_\E$ 
may not satisfy this integrability condition, the Lyapunov splitting and the Lyapunov
 metric are obtained in this case from the results for the full differential $Df$. 

The centralizer part (v) will follow directly from (3) of Theorem \ref{NFext} since 
$X'=\cap_{n\in \Z}\, g^n(X)$ is the desired invariant set of full measure as
$g$ preserves the measure class of $\mu$.
Moreover, $g(W_x)=W_{gx}$ since $g$ is a diffeomorphism commuting with $f$, 
so that $X'$ is also saturated by the stable manifolds.

Parts (i), (ii), (iii) essentially follow from Theorem \ref{NFext}, which is formulated 
so as to apply to this setting. First we consider the regular set $\La$. We fix a 
family of local (strong) stable manifolds $W_{x,r(x)}$ for $x \in \La$  of sufficiently 
small size $r(x)$. Identifying $W_{x,r(x)}$ by an exponential map with a neighborhood of $0$ in $\E_x$ we obtain the extension $\f=\{\f_x\}$ of $f$. 
Then the properties of local stable manifolds ensure that $\f$ satisfies 
the assumptions of Theorem \ref{NFext}. Indeed, they are given by $\Cr$ embeddings so that the $\Cr$ norm and $1/r(x)$ are $\e$-tempered for 
any $\e>0$ (see \cite{BP} for a general reference and \cite[Theorem 5]{KtR} 
for a convenient statement of the stable manifold theorem). Hence 
Theorem \ref{NFext} yields existence of the desired family of  local
diffeomorphisms $\h_x$, $x\in \La$, which can be uniquely 
extended to global diffeomorphisms by invariance. 

Now we define $X=\cup_{x\in \La} W_x$ and explain the construction 
of $\h_y$ for any $y\in X$. By iterating it forward we may assume that 
$y\in W_{x,r(x)}$.
While the individual Lyapunov spaces $\E^i$ may 
not be defined for all points $y\in W_{x,r(x)}$, the flag $\V$ of fast subspaces 
\begin{equation}\label{fastflag}
\E_x^1=\V_x^1 \subset \V_x^2 \subset ... \subset \V_x^l =\E_x, \quad \text{where }\; \V^i_x= \E^1_x \oplus \dots \oplus \E^i_x,
\end{equation} 
is defined for each $\E_y=T_y W_{x,r(x)}$. Moreover, the subspaces $\V^i_y$ 
depend H\"older continuously, and in fact $C^{N -1,\a}$, on $y$ along 
$W_x$\, \cite[Theorem 6.3]{R}. 

The key observation is that the notion of sub-resonance polynomial 
depends only on the fast flag $\V$ \cite[Proposition 3.2]{KS15}, not on the
individual Lyapunov spaces $\E^i$, and thus is well-defined for $\E_y$. 
Then the sub-bundle $\s^{(n)}$ of sub-resonance polynomials of degree $n$
is well-defined, invariant under $Df$, and H\"older continuous in $y$ along $W$, 
and hence so is the factor bundle $\po^{(n)} / \s^{(n)}$. 
Then for each $y\in W_{x,r(x)}$  we can define $\h_y$ using the construction 
in  Theorem \ref{NFext}. Indeed, first we constructed the Taylor term of degree 
$n$ using the contraction $\tilde \Phi$ on the bundle 
$\po^{(n)} / \s^{(n)}$ from Lemma \ref{contraction} with linear part estimated as
$\| \Phi_x (R) \|_{\e,x} \le e^{\la +(n+1)\e}\| R \| _{\e,fx}$. Then $\Phi_y$, the corresponding map at $y$ is H\"older close to $\Phi_x$. Using the Lyapunov 
norm at $x$ as the reference norm, we obtain that $\Phi_y$ is also
a contraction with similar estimate for all $y\in W_{x,r(x)}$ provided that $r(x)$ is sufficiently small. Since $f^ky \in W_{f^kx,r(f^kx)}$ by the contraction 
property of $W_{x,r(x)}$, the closeness persists along the forward trajectory.
This argument is similar to the proof of Lemma \ref{4.1}.
Then we obtain that the operator $\tilde \Phi_y$ on the sequence space is also a contraction.
Thus we can define $\bar \h_y^{(n)}$ as before using the unique fixed 
point in the space of sequences. The last step of the construction can be carried out
similarly as it involves only the estimates of the derivatives on the full space $\E$
and does not depend on the splitting.
\begin{remark}
{\em 
Any measurable choice of transversals $\tilde \E^i$ to $\V^{i-1}$ inside $\V^i$,
$i=2,...,\ell$, yields a transversal $\tilde \n^{(n)} $ to $\s^{(n)}$ inside $\po^{(n)}$.
The latter
gives a preferred choice of the lift. The fixed point of the contraction $\bar \h_y^{(n)}$ 
depends H\"older continuously (and even smoothly by appropriate $C^r$ section 
theorem as in \cite{KS15}) on $y$ along $W_{x,r(x)}$ if the same holds for the 
data $\tilde Q$ obtained in the previous step of the construction. To complete 
the inductive step we need a H\"older lift $\h_y^{(n)}$ to $\po^{(n)}$. 
If there is a consistent choice which is H\"older  on the full leaves of $W$, then
we can obtain a family $\{ \h_x\}$ which is H\"older along the leaves of $W$. 
In contrast to the uniform setting of \cite{KS15},  it is not 
clear that such a choice exists. However, this can be done locally on 
$W_{x,r(x)}$, so then one can fix a Ledrappier-Young partition subordinate 
to the leaves of $W$ and obtain H\"older continuity of $\h_x$ on each element. 
}
\end{remark}



\vskip.2cm

\subsection{Consistency of the fast foliations.} \label{foliations}

The leaf $W_{x}$ is subfoliated by unique foliations $U^{k}$ tangent to $\V^k_y$.
We denote by $\bw^{k}$ the corresponding foliations of $\E_x$ obtained by
the identification $\h _x :W_x \to \E_x$. Thus we obtain the foliations 
$\bw^{k}$ of $\E$ which are invariant under the polynomial extension
$\p$. Since the maps $\h _x$ are diffeomorphisms, $\bw^{k}$ are 
also the unique fast foliations with the same contraction rates. They are 
characterized, for any $\e$ sufficiently small so that $\chi _k +\e< \chi_{k+1}$, by
$$\text{for } \;y, z \in  \E_x  \qquad z \in \bw^{k} (y) \; \Leftrightarrow \; \dist (\p^n _x(y),\p^n_x (z)) \le C e^{n (\chi _k +\e)} \text{ for all } n \in \N.$$

It follows from  Definition \ref{SRdef} that sub-resonance polynomials 
$R \in \s_{x,y}$ are {\em block triangular} in the sense that $\E^i$ component 
does not depend on $\E^j$ components for $j<i$ or, equivalently, it maps
map the subspaces $\V^i_x$ of fast flag in $\E_x$ to those in $\E_y$. 

It is easy to see that all derivatives of a sub-resonance  polynomial are 
sub-resonance polynomials. 
In particular, the derivative 
$D_y \p_x$ at {\em any} point $y \in \E_x$ is sub-resonance and hence is 
block triangular. Thus it maps subspaces parallel to $\V^{k}_x$ to subspaces 
parallel to $\V^{k}_{fx}$. Hence the foliation of $\E$ by subspaces parallel to 
$\V^{k}_x$ in $\E_x$ is invariant under the extension $\p$ and hence coincides 
with $\bw^{k}$ by uniqueness of the fast foliation. 
\begin{remark}
This implies that the fast subfoliations $U^{k}$ are as smooth along the 
leaf $W_x$ as the diffeomorphism $\h_x$ which maps them 
to linear subfoliations of $\E_x$. 
\end{remark}

\noindent
It follows that for any $x\in \M$ and any $y\in W_x$  the diffeomorphism
  \begin{equation}\label{hdef}
\g_{x,y}:=\h_y \circ \h_x^{-1} : \E_x \to \E_y
\end{equation}
maps the fast flag of linear foliations of $\E_x$ to that of $\E_y$.
In particular, the same holds for its derivative 
$ D_0 \g_{x,y}=D_x \h_y : \E_x \to \E_y$ and we conclude that 
$ D_0 \g_{x,y}$ is block triangular and thus is a sub-resonance 
linear map.


\vskip.2cm 

\subsection{Proof of (iv): Consistency of normal form coordinates.}

We need to show that the map $\g_{x,y}$ in \eqref{hdef}
is a sub-resonance polynomial map for all $x\in X$ and $y \in W_x$. 
It suffices to consider $x\in \La$ and, using invariance, we may assume that 
$y \in W_x$ is sufficiently close to $x$. 
 First we note that
$$
\g_{x,y} (0)  =\h_y (x)  =:\bar x \in \E_y \quad\text{and}\quad
 D_0 \,\g_{x,y}=D_x \h_y.
$$ 
Since $ \h_{f^n x}^{-1}  \circ \p ^n_x  \circ \h_{x} = f^n = \h_{f^n y}^{-1}  \circ \p ^n_y  \circ \h_{y}$ we obtain that
$$  \h_{f^n y} \circ \h_{f^n x}^{-1}  \circ \p ^n_x  
= \h_{f^n y} \circ f^n \circ \h_x^{-1} = 
  \p ^n_y  \circ \h_y \circ \h_x^{-1}  \quad \text {and hence}
  $$ 
  \begin{equation}\label{hcom}
\g_{f^n x,f^n y} \circ \p ^n_x =  \p ^n_y  \circ \g_{x,y}.
\end{equation}
Now we consider the Taylor polynomial  for $\g_{x,y} : \E_x \to \E_y$ at $t=0 \in \E_x$: 
$$\g_{x,y}(t) \sim G_{x,y}(t)=\bar x + \sum_{m=1}^N G^{(m)}_{x,y}(t). 
$$
Our first goal is to show that all its terms are sub-resonance polynomials.
We proved in Section \ref{foliations} that the first derivative 
$ G^{(1)}_{x,y}=D_x \h_y $ is a  sub-resonance linear map. 

Inductively, we assume that $ G^{(m)}_{x,y}$ has only sub-resonance 
terms for $m=1,...,k-1$ and show that the same holds for $ G^{(k)}_{x,y}$. Suppose for the contrary that $G^{(k)}_{x,y}$ is not a sub-resonance polynomial 
 and consider order $k$ terms in 
the Taylor polynomial at $0\in \E_x$ for \eqref{hcom}. 
Taylor polynomial  for $\p _x^n$ at 
$0\in \E_x$ coincides with $\p _x^n(t)=\sum_{m=1}^d \pd ^{(m)}_{x}(t)$. 
We also consider the Taylor polynomial  for 
$\p _y^n$ at $\g_{x,y}(0) = \bar x \in \E_y$ 
$$
\p _y^n(z)= \bar x_n + \sum_{m=1}^d Q^{(m)}_{y}(z-\bar x), \quad 
\text{where } \bar x_n =\p _y^n (\bar x). 
$$
All terms $Q^{(m)}$ are sub-resonance as the derivatives of a sub-resonance  polynomial.
Consider the Taylor polynomial   for 
$$\g_{f^n x,f^n y}(t) \sim G_{f^n x,f^n y}(t) = \bar x_n  + \sum_{m=1}^N G^{(m)}_{f^n x,f^n y}(t).
$$
Now we obtain from \eqref{hcom} the coincidence of the terms up to degree $N$ in 
$$\bar x_n  + \sum_{j=1}^N G^{(j)}_{f^n x,f^n y}\left( \sum_{m=1}^d \pd ^{(m)}_{x}(t) \right) \sim \,\bar x_n  + 
\sum_{m=1}^d Q^{(m)}_{y}\left( \sum_{j=1}^N G^{(j)}_{x,y}(t)  \right).
$$
Since any composition of sub-resonance polynomials is again sub-resonance, 
the inductive assumption gives that all terms of order $k$ in the above equation must be sub-resonance polynomials except for  
$$
G^{(k)}_{f^nx,f^ny} \left(  \pd ^{(1)}_{x}(t) \right) \quad \text {and}  \quad Q^{(1)}_{y}\left( G^{(k)}_{x,y}(t)  \right). 
$$
Multiplying these terms  on the left by sub-resonance linear map $\left(D_0 G^{(k)}_{f^nx,f^ny}\right)^{-1}= \left(D_{f^n x}\h_{f^n y}\right)^{-1}$ and using the fact 
that $\pd ^{(1)}_{x}=F^n_x=Df^n |_{\E_x}$ and
$$Q^{(1)}_{y} =D_{\bar x} \p _y^n = D_{f^n x}\h_{f^n y} \circ \fd^n_x \circ (D_x \h_y) ^{-1} 
$$ 
we obtain that the following maps from $\E_x$ to $\E_{f^nx}$ agree modulo 
sub-resonance terms
$$
\left(\left(D_{f^n x}\h_{f^n y}\right)^{-1} \circ G^{(k)}_{f^nx,f^ny}\right) \circ  F^n_x  \cong  \fd^n_x \circ  \left( (D_x \h_y) ^{-1} \circ G^{(k)}_{x,y}  \right) 
 \mod  \S_{x,f^nx}.
$$
Since $x, f^n x \in \La$ and thus the spaces $\E_x$ and $\E_{f^nx}$ have
Lyapunov splittings we can decompose these polynomial maps into 
sun-resonance and non sub-resonance terms.
Taking non sub-resonance terms on both sides we obtain the equality
  \begin{equation}\label{NRcom}
N_{f^n x}  \circ  F^n_x  =  \fd^n_x \circ N_{x} 
\end{equation}
where $N_{f^n x}$ and $N_{x}$ denote the non sub-resonance terms in
$\left(D_{f^n x}\h_{f^n y}\right)^{-1} \circ G^{(k)}_{f^nx,f^ny}$ and 
$ (D_x \h_y) ^{-1} \circ G^{(k)}_{x,y}$ respectively. If the latter had only 
sub-resonance terms then so would $G^{(k)}_{x,y}$, contradicting the 
assumption. Hence $N_{x}\ne0$. 
We decompose $N_x$ into components $N_x=(N_x^1,... , N_x^{\ell})$ and let $i$ be the largest index so that  $N^i _x \ne 0$, i.e. there exists $t' \in \E_x$ so that $z'= N (t')$ has non-zero component in $\E^i_y$, which we denote by $z'_i$. Then by \eqref{estAEi} we obtain
 \begin{equation}\label{right}
\| \fd^n_x \circ N_{x} (t') \|_{f^nx} = \| \fd^n_x (z)\|_{f^nx}   \ge e^{n (\chi_i -\e)}  \| z'_i \|_x.
 \end{equation}
 \vskip.1cm
 
 Now we estimate the norm of the $i$ component of the left-hand side of \eqref{NRcom} at $t'$. 
 For each componet $t'_j$ of $t'$  we have 
 $ \| \fd^n_x (t'_j) \|_{f^nx}\le  e^{n  (\chi _j +\e)}  \| t'_j \|_x $  by \eqref{estAEi}.
Let $N^s_{f^n x}$ be a  term of homogeneity type $s=(s_1,  ..., s_{\ell})$ in the component $N_{f^n x}^i$.  Then we obtain as in Lemma \ref{exponents}
 $$
 \| N^s_{f^n x} \left(  \fd^n_x(t')\right) \|_{f^nx} \le 
\| N_{f^n x}\|_{f^nx} \cdot  \| t' \|_x^k \cdot  e^{n \sum s_j (\chi _j +\e)} .
$$
Since no term in $N_{f^n x}^i$ is a sub-resonance one, we have $\chi _i > \sum s_j \chi _j$.  
  This contradicts \eqref{NRcom} and \eqref{right} for large $n$ if $\e$ is sufficiently small since $\| N_{f^n x}\|_{f^nx}$ is tempered. 
 The latter follows from temperedness of $G^{(k)}_{f^nx,f^ny}$ and the 
fact that $D_{f^n x}\h_{f^n y}$ is H\"older close to the identity and so the norm
of its inverse is bounded in Lyapunov metric.

\vskip.1cm

We conclude that  for all $x\in X$ and $y \in W_x$ the Taylor polynomial 
$G_{x,y}$  of $\,\g_{x,y}$ contains only sub-resonance  terms. 
Now we will show that $\g_{y,x}$ coincides with its Taylor polynomial.
Again it suffices to consider $x\in \La$ and $y \in W_x$ which is 
sufficiently close to $x$. In addition to \eqref{hcom} we have the same relation for their Taylor
polynomials
 \begin{equation}\label{Hcom}
G_{f^n y,f^n x} \circ \p ^n_y =  \p ^n_x  \circ G_{y,x}.
\end{equation}
Indeed, the two sides must have the same terms up to order $N$, 
but these are sub-resonance polynomials and thus have no terms of 
degree higher than $d\le N$.

Denoting $\Delta_n= \g_{f^n y,f^n x}-G_{f^n y,f^n x}$ we obtain from
 \eqref{hcom} and \eqref{Hcom} that
 \begin{equation}\label{Hcom2}
\Delta_n \circ \p ^n_y   =  \p ^n_x  \circ \g_{y,x} -  \p ^n_x  \circ G_{y,x} .
\end{equation}

We denote $\Delta= \g_{y,x}-G_{y,x} :\E_y \to \E_x$ and suppose 
that $\Delta \ne0$. 
Let $i$ be the largest index for which the $i$ component of $\Delta$ is 
nonzero. Then there exist arbitrarily small $t' \in \E_y$  such that the $i$ component
$z'_i$ of  $z'=\Delta (t') $ is nonzero. Since $\p ^n_x$ is a sub-resonance 
polynomial, the nonlinear terms in its $i$ component can depend only
on $j$ components of the input with $j>i$, which are the same for
$\g_{y,x}$ and $G_{y,x}$. Thus the $i$ component of the right side of 
\eqref{Hcom2} is $F^n_x (z'_i)$ since the linear part of $\p ^n_x$ is
$F^n_x$ and it preserves the Lyapunov splitting. So by \eqref{estAEi}
we can estimate  the right side of \eqref{Hcom2} 
\begin{equation}\label{HcomR}
\| \left( \p ^n_x  \circ \g_{y,x} -  \p ^n_x  \circ G_{y,x}\right)(t') \|_{f^nx} \ge 
\| F^n_x (z'_i)\|_{f^nx} \ge e^{n  (\chi _i - \e)}  \| z'_i \|_x \ge e^{n  (\chi _1 -\e)}  \| z'_i \|_x .
\end{equation}

Now we estimate  the left side of \eqref{Hcom2}.
Since $\g_{f^n y,f^n x}$ is $\Cr$  there exists 
$C_{n}$ determined by $\| \g_{f^n x,f^n y}\|_{\Cr} $ such that 
\begin{equation}\label{Ck,delta}
\| \Delta_n (t)\| \le C_{n} \| t\| ^{N+\a} \quad\text{for all }
t \in  \E_{f^n x}  \text{ with } \| t\| \le r_n.
\end{equation}

To estimate $\p ^n_y $ we note that $D_0 \p^n_{y}=F^n_y=Df^n |_{\E_y}$ 
and its norm for $y$ close to $x$ can be estimated using Lemma \ref{4.1}(3).
Then $\p ^n_y $ itself can be estimated as in that lemma:
$$\| \p ^n_y  (t) \| \le K e^{n  (\chi _{\ell} +3\e)} \| t\| $$
for all sufficiently small $t \in  \E_{y}$. Combining this with \eqref{Ck,delta}
we obtain
$$
\|\left( \Delta_n \circ \p ^n_y \right) (t') \|\,\le \, C_{n} \|  \p ^n_y  (t')\| ^{N+\a}
\le C_n (K \| t'\| )^{N+\a} e^{n (N+\a) (\chi _{\ell} +3\e)}.
$$
Now we see that this contradicts \eqref{Hcom2} and \eqref{HcomR} for
large $n$ if $\e$ is sufficiently small. Indeed $(N+\a) \chi _{\ell}<\chi _1$
while $C_n$ is tempered and  the Lyapunov norm satisfies
$\| u \| \ge K(x) e^{-n\e}\| u \|_{f^nx}$.
Thus, $\Delta=0$, i.e. the map $\g_{y,x}$ coincides with its Taylor polynomial.
\vskip.1cm

This completes the proof of Theorem \ref{NFfol}.
$\QED$
\vskip.05cm

\subsection {Proof of Corollary \ref{1/2pinch}}
If $d=1$ then all sub-resonance polynomials are linear, the maps 
$\h_y \circ \h_x^{-1} : \E_x \to \E_y$ are affine, and the family 
$\{ \h_x \} _{x\in X}$ is unique by part (2) of Theorem \ref{NFext}.
If we identify $W_x$ with $\E_x$ by $\h_x$, then $\h_y$  for $y\in W_x$ becomes an affine map $\E_x \to T_y \E_x$ with identity differential and
$\h_y(y)=0$. Thus it depends $C^N$ on $y$ as  the coordinate 
system $\h_x$ is $C^N$.



\begin{thebibliography}{99}


\bibitem[BP]{BP}  L. Barreira and Ya. Pesin.  Nonuniformly Hyperbolicity: Dynamics of Systems with Nonzero Lyapunov Exponents. Encyclopedia of Mathematics and
Its Applications, {\bf 115} Cambridge University Press.

\bibitem[BKo]{BK} I. U. Bronstein and A. Ya. Kopanskii. 
Smooth invariant manifolds and normal forms. World Scientific, 1994.

\bibitem[F07]{Fa} Y. Fang. On the rigidity of quasiconformal Anosov flows. Ergodic Theory Dynam. Systems 27 (2007), no. 6, 1773-1802.

\bibitem[FFH10]{FFH} Y. Fang, P. Foulon, and B. Hasselblatt. Zygmund strong foliations in higher dimension. J. Mod. Dyn. 4 (2010), no. 3, 549-569.


\bibitem[Fe04]{F2}  R. Feres.  A differential-geometric view of normal forms 
of contractions. In Modern Dynamical Systems and Applications, Eds.: M. Brin, B. Hasselblatt, Y. Pesin, Cambridge University Press, (2004) 103-121.

\bibitem[FKSp11]{FKS} D. Fisher, B.  Kalinin, and R.  Spatzier.  Totally nonsymplectic Anosov actions on tori and nilmanifolds. Geom. Topol. 15 (2011), no. 1, 191- 216.


\bibitem[GoKS11]{GKS10}  A. Gogolev, B. Kalinin, and V. Sadovskaya.  Local rigidity for Anosov automorphisms.
(with Appendix by R. de la Llave)   
Mathematical Research Letters, 18 (2011), no. 05, 843-858.

\bibitem[G02]{G}   M. Guysinsky. The theory of non-stationary normal forms.
              Ergod. Theory  Dyn. Syst., {\bf 22} (3), (2002), 845--862.

\bibitem[GKt98]{GK}  M. Guysinsky and  A. Katok. Normal forms and invariant 
              geometric structures for dynamical systems with invariant 
              contracting foliations. Math. Research Letters {\bf 5}
              (1998), 149-163.
   

\bibitem[KKt01]{KKt00} B. Kalinin and A. Katok. Invariant measures for actions of higher
rank abelian groups. Proceedings of Symposia in Pure Mathematics. Volume {\bf 69},
(2001), 593-637.

\bibitem[KKt07]{KKt} B. Kalinin and A. Katok.  Measure rigidity beyond uniform 
            hyperbolicity: invariant measures for Cartan actions on tori. 
               Journal of Modern Dynamics, { Vol. 1} (2007),  no. 1, 123-146.
   
\bibitem[KKtR11]{KKtR} B. Kalinin, A. Katok, and F. Rodriguez-Hertz. 
         Nonuniform measure rigidity. Annals of Mathematics 174 (2011), no. 1, 361-400.                

\bibitem[KtR15]{KtR} A. Katok and F. Rodriguez-Hertz. 
Arithmeticity and topology of higher rank actions of Abelian groups.    
         Journal of Modern Dynamics, { Vol. 10} (2016) 115-152.   
       
\bibitem[KtL91]{KL}  A. Katok and J. Lewis. Local rigidity for certain groups of 
              toral automorphisms. Israel J. Math. {\bf 75} (1991), 203-241.
              
\bibitem[KtSp97]{KSp97} A. Katok and R. Spatzier. 
     Differential rigidity of Anosov actions of higher rank abelian groups
    and algebraic lattice actions. 
     Tr. Mat. Inst. Steklova 216 (1997), Din. Sist. i Smezhnye Vopr.,
   292Ð319; translation in Proc. Steklov Inst. Math. 1997, no. 1 (216), 287-314.

\bibitem[KS03]{KS03} B. Kalinin and V. Sadovskaya. On local and global rigidity 
of quasiconformal Anosov diffeomorphisms. Journal of the Institute of Mathematics 
of Jussieu (2003) { 2} (4), 567-582.
              
 \bibitem[KS06]{KS} B. Kalinin and V. Sadovskaya. 
 Global rigidity for totally nonsymplectic Anosov $\Z^k$ actions. 
  Geometry and Topology, vol. 10 (2006), 929-954.
  
               
\bibitem[KS15]{KS15} B. Kalinin and V. Sadovskaya. Normal forms on contracting foliations: smoothness and homogeneous structure. To appear in Geometriae Dedicata.

\bibitem[LL05]{LL} W. Li and  K. Lu. Sternberg theorems for random dynamical systems. Communications on Pure and Applied Mathematics, 
Vol. LVIII  (2005), 0941-0988.

\bibitem[M16]{M} K. Melnick.
Nonstationary smooth geometric structures for contracting measurable cocycles.
 http://www.math.umd.edu/~kmelnick/

  

\bibitem[R79]{R} D. Ruelle. Ergodic theory of differentiable dynamical systems. Publications Math\'ematiques de l'I.H.\' E.S. 50 (1979), 27-58.


\bibitem[S05]{S} V. Sadovskaya. On uniformly quasiconformal Anosov systems. 
Math. Research Letters, vol. 12 (2005), no. 3, 425-441.

\bibitem[St57]{St} S. Sternberg. Local contractions and a theorem of Poincar\'e. Amer. J. of Math. 79 (1957),
809-824. 

\end{thebibliography}
\end{document}